\documentclass[reqno, 10pt]{amsart}
\usepackage{amsmath,amssymb,comment}
\usepackage{enumerate}
\usepackage{hyperref}

\newtheorem{thm}{Theorem}[section]
\newtheorem{prop}[thm]{Proposition}
\newtheorem{cor}[thm]{Corollary}
\newtheorem{lem}[thm]{Lemma}

\theoremstyle{remark}
\newtheorem{defi}[thm]{Definition}
\newtheorem{rem}[thm]{Remark}
\newcommand{\mbtx}{\tilde{\mb x} }

\newcommand{\tr}{\tilde r }
\newcommand{\tx}{\tilde x}
\newcommand{\tsig}{\tilde\sigma}

\newcommand{\al}{\alpha}

\newcommand{\Del}{\Delta}
\newcommand{\del}{\delta}
\newcommand{\lam}{\lambda}
\newcommand{\sig}{\sigma}
\newcommand{\pt}{\partial}

\newcommand{\floor}[1]{\left\lfloor #1 \right\rfloor}
\newcommand{\ceil}[1]{\left\lceil #1 \right\rceil}

\newcommand{\R}{\mathbb R}

\newcommand{\mb}{\mathbf}

\makeatletter
\newcounter{savesection}
\newcounter{apdxsection}
\renewcommand\appendix{\par
	\setcounter{savesection}{\value{section}}%
	\setcounter{section}{\value{apdxsection}}%
	\setcounter{subsection}{0}%
	\gdef\thesection{\@Alph\c@section}}
\newcommand\unappendix{\par
	\setcounter{apdxsection}{\value{section}}%
	\setcounter{section}{\value{savesection}}%
	\setcounter{subsection}{0}%
	\gdef\thesection{\@arabic\c@section}}
\makeatother

\title[Explicit formula for limiting spherical heat kernels]{An explicit formula for the limit of the heat kernel measures on the spheres embedded in $\R^\infty$}

\author{Minh-Luan Doan}
\address{University of Connecticut}
\email{luan.doan@uconn.edu}

\author{Evan O'Dorney}
\address{Institute for Defense Analyses, Center for Communications Research - Princeton}
\email{emo916math@gmail.com}

\date{November 2026}
\thanks{Special thanks to Professor Brian Hall (University of Notre Dame) for his valuable suggestions that lead to the great improvement of this project.}

\begin{document}
	\maketitle
	
	\begin{abstract}
		We show that the heat kernel measures based at the north pole of the spheres $S^{N-1}(\sqrt N)$, with a properly scaled radius $\sqrt N$ and adjusted center, in the sense of finite-dimensional marginals, to a Gaussian measure on $\R^\infty$. In particular, in the direction of the base point, the variance of this limiting measure is $1-e^{-t}-te^{-t}$, and in the orthogonal directions, the variance is $1-e^{-t}$. 
	\end{abstract}
	
	\section{Introduction and Statement of the Main Theorem}
	\subsection{Introduction}
	For any $a>0$, let $S^{N-1}(a)=\{\mb x\in\R^N\colon |\mb x|=a\}$ be the $(N-1)$-sphere of radius $a$, and let $d\sig^{N-1}_a$ denote its rotation-invariant normalized volume measure. If $a$ is fixed, letting $N\to\infty$ does not yield a meaningful limiting measure for $d\sig^{N-1}_a$, since there is no rotation-invariant probability measure on the infinite-dimensional sphere $S^\infty(a)=\{\mb x\in\ell^2\colon|\mb x|=a\}$.\footnote{If $\mu$ were such a measure, rotation-invariance implies $\int x_n^2\,d\mu=c$ for all $n$. Then, $a^2=\int|\mb x|^2\,d\mu=\sum_n\int x_n^2\,d\mu=\sum_nc<\infty$ forces $c=0$, so $\mb x=0$ $\mu$-almost surely, a contradiction to $a=|\mb x|>0$.} The remedy, often found in statistical mechanics, is to scale the radius with the dimension, usually $a=\sqrt N$, and to regard $S^{N-1}(\sqrt N)$ as embedded in $\R^\infty$.

	A common way to describe the results from this scaling is as follows. 
	Recall that the normal probability measure on $\R$ with mean $0$ and variance $s$ is
	$$d\mu^1_s(x)=\mu^1_s(x)\,dx=(2\pi s)^{-1/2}e^{-{x^2}/{(2s)}}\,dx.$$
	The product of $k$ copies of $d\mu^1_s$ is given by
	$$d\mu^k_s(x_1,\dots,x_k)=(2\pi s)^{-k/2}\exp\left\{-\dfrac{x_1^2+x_2^2+\dots+x_k^2}{2s}\right\}\,dx_1\,\dots\,dx_k.$$
	The infinite product $d\mu^\infty_s$ (in the sense of the Kolmogorov Extension Theorem) is characterized by the compatibility property
	\begin{align*}
		\int_{\R^\infty} f(\mb x)\,d\mu^\infty_s(\mb x) = \int_{\R^k} f(\mb y)\,d\mu^k_s(\mb y)
	\end{align*}
	for every (sufficiently nice) function $f$ of $k$ variables $x_1,\dots,x_k$ only. The classical statement, usually attributed to Maxwell, Borel, and Poincar\'e, which goes back at least to Mehler \cite{Mehler1866} and was popularized in statistical mechanics by Boltzmann \cite{Boltzmann} (see Diaconis and Freedman \cite{DiaconisFreedman87} for the history), says that the normalized spherical measures $d\sig^{N-1}=d\sig^{N-1}_{\sqrt N}$ on $S^{N-1}(\sqrt N)$ converge to the standard Gaussian measure $d\mu^\infty_1$. Rigorous proofs are provided by Hida and Nomoto \cite{HiNo64} and Umemura and K\^ono \cite{UmKo65}. 
	
	\medskip
	
	One of the simplest explicit formulations by Peterson and Sengupta \cite[Thm.~2.1]{PeSen} states the following:
	
	\begin{thm}\label{thm:limitsph} 
		For any \emph{polynomial} $f$ of fixed $k$ variables $x_1,\dots, x_k$, we have 
		\begin{align*}
			\lim_{N\to\infty} \int_{S^{N-1}(\sqrt N)} f(x_1,\dots,x_k)\,d\sig^{N-1}(\mb x)
			&= \int_{\R^k} f(x_1,\dots,x_k)\,d\mu^k_1 (\mb x).
		\end{align*}
	\end{thm}

	\medskip
	
	
	To extend this classical result, one can regard the spherical measure $d\sig^{N-1}$ as the limit of the heat kernel measures on $S^{N-1}(\sqrt N)$ as $t\to\infty$. Indeed, if $X_t$ is a Brownian motion on $S^{N-1}(\sqrt N)$, with generator $\frac12\Del_{S^{N-1}(\sqrt N)}$, then the law of $X_t$ started at a point $\mb p$ is the heat kernel measure, and this measure converges to $d\sig^{N-1}$ as $t\to\infty$. Hence, the natural question is what happens at a finite time $t>0$. Heat kernels on infinite-dimensional flat Hilbert spaces are problematic (see, for example, Driver \cite{Driver2003}), but in the setting of high-dimensional spheres embedded in $\R^\infty$, one can expect a Gaussian limit.

	\medskip
	
	Several related works on this question are known. Takatsu \cite{Takatsu2021} proved that the Laplace operators and heat semigroups of the spheres converge to those of Gaussian space, while Umemura and K\^ono \cite{UmKo65} and Peterson and Sengupta \cite{PeSen} analyzed the spherical harmonics and polynomials underlying this limit. 
	Skipper \cite[Thm.~2.1]{Skipper2010} shows that if $\boldsymbol\theta^{(N)}$ is a unit vector and the initial position $\mb x^{(N)}\in S^{N-1}(\sqrt N)$ satisfies $\langle\boldsymbol\theta^{(N)},\mb x^{(N)}\rangle\to u$, then the projection $\langle\boldsymbol\theta^{(N)},X_s\rangle$ of spherical Brownian motion converges weakly, as a process, to an Ornstein--Uhlenbeck process started at $u$ (his generator is $\Del_{S^{N-1}(\sqrt N)}$ without the factor $\frac12$, so his time $s$ corresponds to $t=2s$ in this paper). Ribeiro and Elskens \cite[Thms.~4.1, 6.1]{RibeiroElskens2013} study the same diffusion, with the same time normalization as in this paper, and prove mean-square convergence in $C([0,T])$ of finitely many coordinates, started at fixed values, to Ornstein--Uhlenbeck processes driven by the same Brownian motions. Cutland \cite[Thms.~1.1, 5.1]{Cutland1999} approaches differently using nonstandard analysis and proves that the partial-sum process $\sum_{i\le k}x_i^{(n)}(t)$ of Brownian motion on $S^{n-1}(1)$ converges weakly to the infinite-dimensional Ornstein--Uhlenbeck process on Wiener space, provided the initial partial-sum paths converge weakly in $C[0,1]$.

	\medskip
	
	
	In the hypotheses of the results above, the north pole $\mb p=(\sqrt N,0,\dots,0)$, which is the most natural base point of a heat kernel. 
	In particular, the error bound (28) in \cite{RibeiroElskens2013}  tends to zero only if $|c_1|^2/N\to0$ for the initial value $c_1$ of the coordinate. The hypothesis in \cite{Cutland1999} on the initial partial-sum paths excludes the north pole, whose partial-sum path jumps at $s=0$. At the north pole, $\langle\mb e_1,\mb p\rangle=\sqrt N$, and the first coordinate has certain mean $m(t,N)\sim\sqrt N\,e^{-t/2}$ defined in \ref{sec:notations} (see Proposition~\ref{prop:mean-x1}), which diverges as $N\to\infty$. The difference is significant. An Ornstein--Uhlenbeck process started at $\sqrt N$ would predict the limiting variance $1-e^{-t}$ for the first coordinate, 
	whereas its exact variance under the heat kernel is $1+(N-1)e^{-t}-Ne^{-t(1-1/N)}$, which tends to $1-e^{-t}-te^{-t}$. 
	This is because the diffusion coefficient $1-x_1^2/N$ of $\pt_{x_1}^2$, which is $1+O(1/N)$ when $x_1=O(1)$, is approximately $1-e^{-t}$ at distance $m(t,N)\sim\sqrt N\,e^{-t/2}$ from the origin, an effect that 
	is not seen in the known results.

	\medskip
	
	In this work, we address the limit of the heat kernel measures based at the north pole. Once the sphere is recentered by $-m(t,N)\mb e_1$, the measures converge, for every fixed $t>0$ and jointly in the first $k$ coordinates, to an explicit centered Gaussian (Theorem~\ref{thm:main}) whose variance in the pole direction is $1-e^{-t}-te^{-t}$, and in each orthogonal direction with independent coordinates, $1-e^{-t}$. 
	
	By rotation invariance the result can be adapted to any base point $\mb p=\sqrt N\,\omega'$ whose direction $\omega'$ has \emph{finitely many} nonzero coordinates (Corollary~\ref{cor:basepoint}). The limiting covariance is $(1-e^{-t})I_k-te^{-t}\omega\omega^{\mathsf T}$, where $\omega \in\R^k$ denotes the first k coordinates of $\omega'$. In particular, the overlap $\langle\mb e_1,\mb p\rangle\sim\alpha\sqrt N$, with $0\le\alpha\le1$, gives variance $1-e^{-t}-\alpha^2te^{-t}$, unifying the Ornstein--Uhlenbeck picture of \cite{Skipper2010} ($\alpha=0$) and the pole ($\alpha=1$). 
	Since a Gaussian is determined by its moments, the convergence also holds weakly and in every Wasserstein distance (Corollary~\ref{cor:weak}), and, as $t\to\infty$, the limit becomes the standard Gaussian measure in Theorem~\ref{thm:limitsph}.

	The proofs in this work are elementary but are of interest in their own right. A heuristic derivation from the heat equation (Section~\ref{sec:heat-poly}) explains where the time-inhomogeneous Ornstein--Uhlenbeck equation with diffusion coefficient $1-e^{-t}$ comes from. We then make it rigorous in weak form: differentiating the heat operator along the moving sphere gives an exact system of equations for polynomial moments, which converges, with an explicit error of order $N^{-1/2}$, to the moment equations of the limiting Ornstein--Uhlenbeck process (Sections~\ref{sec:weak}--\ref{sec:proof-weak}). A second proof, by explicit diagonalization of the radial part of the Laplacian, yields exact finite-$N$ formulas (Section~\ref{sec:proof-eigen}).  Appendix~\ref{apd:pde} provides the solution of the PDE in Section~\ref{sec:heat-poly}, and Appendix~\ref{apd:proof4.1and4.2} provides the combinatorial computations of the key Theorems \ref{thm:x2toxk} and \ref{thm:x1} in Section~\ref{sec:proof-eigen}.

	\subsection{Notations}\label{sec:notations}
	Consider $\R^k$ as a submanifold of $\R^n$ for $k<n$ in the usual way (a point $\mb x=(x_1,\dots,x_k)$ in $\R^k$ can be written as $\mb x=(x_1,\dots, x_k,0,\dots,0)$ in $\R^n$). For $t>0$ and a natural number $N$, define the function
	\begin{align*}
		m(t,N) = \sqrt N \,\exp\left\{\dfrac{t}{2}\left(-1+\dfrac{1}{N}\right)\right\},
	\end{align*}
	and the shifted (time-dependent) first coordinate $\tilde x_1 = \tilde x_1(t,N)=x_1+m(t,N)$. We also use the vector notations
	\begin{align*}
		\tilde {\mb x} = \tilde {\mb x} (t,N)= (\tilde x_1,x_2,\dots, x_N).
	\end{align*}
	
	\medskip
	
	For the statement of the main result, we consider the family of (shifted) spheres $\tilde S^{N-1}(\sqrt N)$ embedded in $\R^N$ defined as follows:
	\begin{align*}
		\tilde S^{N-1} (\sqrt N) = \{\mb x=(x_1,\dots,x_N)\in\R^N\colon \tilde x_1^2+x_2^2+\dots+x_N^2=N\}.
	\end{align*}
	It is an $(N-1)$-dimensional sphere of radius $\sqrt N$ centered at $(-m(t,N),0,\dots,0)$.

	\medskip
	
	According to the general theory of heat equations \cite[Chapter 5]{Davies}, there exists a heat kernel $\tilde\rho^{N-1}_t(\mb x)=\tilde\rho^{N-1}_t(\tilde{\mb p},\mb x)$ at time $t$ on $\tilde S^{N-1}(\sqrt N)$ based at the north pole $\tilde {\mb p} = \tilde {\mb p} (t,N)=(\sqrt N-m(t,N), 0,\dots,0)$ of the sphere. Throughout, $\tilde\rho^{N-1}_t$ is the density with respect to the normalized surface measure $d\tsig^{N-1}$ on $\tilde S^{N-1}(\sqrt N)$ (the translate of $S^{N-1}(\sqrt N)$), so that $\tilde\rho^{N-1}_t(\mb x)\,d\tsig^{N-1}(\mb x)$ is a probability measure, called the heat kernel measure. 
	
	\subsection{Statement of the Main Theorem and corollaries}
	For the rest of the paper, we study the convergence of the heat kernel measure on $\tilde S^{N-1}(\sqrt N)$ when $N\to\infty$ as described in the following theorem.
	\begin{thm}\label{thm:main}
		Fix $t>0$ and a positive integer $k$. Let $\mathcal P^{k}(\R)$ be the vector space of all polynomials of $k$ real variables $x_1, x_2,\dots, x_k$ (not $\tilde x_1$). Then, for any $f$ in $\mathcal P^{k}(\R)$ considered as a function defined on a neighborhood of $\tilde S^{N-1}(\sqrt N)$ embedded in $\R^N$ with $N>k$, we have
		\begin{align*}
			\lim_{N\to\infty} \int_{\tilde S^{N-1}(\sqrt N)}f(x_1,\dots,x_k)\,\tilde\rho^{N-1}_t(\mb x)\,d\tsig^{N-1}(\mb x)= \int_{\R^k} f(x_1,\dots,x_k)\,u^{\infty,k}_t(\mb x)\,d\mb x 
		\end{align*}
		where $u^{\infty,k}_t$ is the density of the marginal on $\R^k$ (image under the projection onto the first $k$ coordinates) of a Gaussian measure (denoted by $u_t^\infty$) on $\R^\infty$, given by
		\begin{align*}
			u^{\infty,k}_t (\mb x) = c_{t,k}\exp\left(-\dfrac{x_1^2}{2(1-e^{-t}-te^{-t})}-\dfrac{x_2^2+x_3^2+\dots+x_k^2}{2(1-e^{-t})}\right)
		\end{align*}
		where $c_{t,k}$ is the normalizing constant
		\begin{align*}
			c_{t,k} =(2\pi)^{-k/2}(1-e^{-t}-te^{-t})^{-1/2}(1-e^{-t})^{(1-k)/2}.
		\end{align*}
	\end{thm}
	
	The following corollary is an immediate consequence of the theorem by rotation invariance:
	\begin{cor}[Change of base point]\label{cor:basepoint}
		Fix $t>0$, a positive integer $k$, and a vector $\omega=(\omega_1,\dots,\omega_k)\in\R^k$ with $|\omega|\le1$. Put
		$\omega'=(\omega_1,\dots,\omega_k,\sqrt{1-|\omega|^2})$, a unit vector, regarded as $(\omega',0,\dots,0)\in\R^N$ for $N>k+1$.
		Let $\mb p_\omega=\sqrt N\,\omega'$ be the base point on $S^{N-1}(\sqrt N)$, let $${\tilde S^{N-1}_\omega(\sqrt N)=\{\mb x\in\R^N:\ |\mb x+m(t,N)\,\omega'|^2=N\}}$$ be the sphere of radius $\sqrt N$ centered at $-m(t,N)\,\omega'$ (shifted in the direction of the base point), and let $\tilde\rho^{N-1}_{t,\omega}$ be the density, with
		respect to the normalized surface measure $d\tilde\sigma^{N-1}_\omega$, of the heat kernel on it based at
		$\tilde{\mb p}_\omega=(\sqrt N-m(t,N))\,\omega'$. Fix $t>0$, then for any polynomial $f\in\mathcal P^k(\R)$,
		\begin{align*}
			\lim_{N\to\infty}\int_{\tilde S^{N-1}_\omega(\sqrt N)}f(x_1,\dots,x_k)\,\tilde\rho^{N-1}_{t,\omega}(\mb x)\,d\tilde\sigma^{N-1}_\omega(\mb x)
			=\int_{\R^k}f(\mb x)\,u_{t,\omega}^{\infty, k}(\mb x)\,d\mb x,
		\end{align*}
		where $u^{\infty,k}_{t,\omega}$ is the centered Gaussian density on $\R^k$ with covariance matrix
		\begin{align*}
			\Sigma^k_\omega(t)=(1-e^{-t})\,I_k-t\,e^{-t}\,\omega\omega^{\mathsf T}.
		\end{align*}
		In particular if $|\omega|=1$, this is the Gaussian with variance $1-e^{-t}-te^{-t}$ in the direction $\omega$ and variance $1-e^{-t}$ in every direction orthogonal to $\omega$, i.e., $u^{\infty, k}_{t,\omega}(\mb x)=u^{\infty,k}_t(Q^{-1}\mb x)$ for any orthogonal matrix $Q\in O(k)$ with $Q\mb e_1=\omega$.
	\end{cor}
	
	\begin{proof}
		Choose $Q_0\in O(k+1)$ with $Q_0\mb e_1=\omega'$ and let $Q=Q_0\oplus I_{N-k-1}\in O(N)$, so $Q\mb e_1=\omega'$ and $Q$ preserves
		$S^{N-1}(\sqrt N)$ and $\sigma^{N-1}$. The heat kernel is invariant under isometries, i.e.,
		$\rho^{N-1}_t(Qa,Qx)=\rho^{N-1}_t(a,x)$, because $\Del_S$ commutes with $Q$. Moreover,
		$Q\tilde S^{N-1}(\sqrt N)=\tilde S^{N-1}_\omega(\sqrt N)$, since $\mb x+m\mb e_1\in S^{N-1}(\sqrt N)$ if and only if
		$Q\mb x+m\omega'\in S^{N-1}(\sqrt N)$, and $Q\tilde{\mb p}=\tilde{\mb p}_\omega$. 
		
		Hence, the heat kernel measure on
		$\tilde S^{N-1}_\omega(\sqrt N)$ based at $\tilde{\mb p}_\omega$ is the image under $Q$ of the one in Theorem~\ref{thm:main}, and, with $\hat f(x_1,\dots,x_{k+1}):=f(x_1,\dots,x_k)$,
		\begin{align*}
			\int_{\tilde S^{N-1}_\omega(\sqrt N)}f\,\tilde\rho^{N-1}_{t,\omega}\,d\tilde\sigma^{N-1}_\omega=\int_{\tilde S^{N-1}(\sqrt N)}(\hat f\circ Q)\,\tilde\rho^{N-1}_t\,d\tilde\sigma^{N-1}.
		\end{align*}
		Since $Q$ acts only on the first $k+1$ coordinates, $\hat f\circ Q$ is a polynomial in $x_1,\dots,x_{k+1}$. Theorem~\ref{thm:main} with $k+1$ in place of $k$
		gives the limit 
		$$\int_{\R^{k+1}}\hat f(Q_0\mb x)\,u^{\infty,k+1}_t(\mb x)\,d\mb x=\int_{\R^{k+1}}\hat f(\mb y)\,u^{\infty,k+1}_t(Q_0^{-1}\mb y)\,d\mb y.$$
		The covariance of this Gaussian is
		$$Q_0\,\mathrm{diag}(1-e^{-t}-te^{-t},\,1-e^{-t},\dots,1-e^{-t})\,Q_0^{\mathsf T}=(1-e^{-t})I_{k+1}-te^{-t}\,\omega'\omega'^{\mathsf T}.$$
		Since $\hat f$ does not depend on $x_{k+1}$, the limit is the integral of $f$ against the marginal of this Gaussian on $\R^k$, which has covariance $(1-e^{-t})I_k-te^{-t}\,\omega\omega^{\mathsf T}=\Sigma^k_\omega(t)$.
	\end{proof}
	
	\begin{rem}
		\begin{enumerate}[(1)]
			\item For $\omega=\mb e_1$ one recovers Theorem~\ref{thm:main}. For $\omega=0$ (the base point $\sqrt N\,\mb e_{k+1}$ is orthogonal to the first $k$ coordinates), the limit is the Ornstein--Uhlenbeck law $N(0,(1-e^{-t})I_k)$, i.e., the fixed-time marginal of the Ornstein--Uhlenbeck process started at $0$, in agreement with Skipper's theorem \cite[Thm.~2.1]{Skipper2010}.
			For $j\le k$ the marginal on $\R^j$ is Gaussian with covariance the upper-left $j\times j$ block of $\Sigma^k_\omega(t)$.
			\item The second moments are exact for every $N$ and every base point $\mb p$: since $\Del_S x_i=-(1-\tfrac1N)x_i$ and
			$\Del_S(x_ix_j)=2\delta_{ij}-2x_ix_j$, one gets $\mathrm{Cov}(x_i,x_j)=\delta_{ij}(1-e^{-t})-p_ip_je^{-t}(e^{t/N}-1)$, which tends to
			the entries of $\Sigma_\omega(t)$ when $\mb p=\sqrt N\,\omega'$.
			\item The hypothesis that the direction $\omega'$ of the base point has only finitely many nonzero coordinates is what makes $\hat f\circ Q$ a polynomial in finitely many variables.
		\end{enumerate}
	\end{rem}

	The convergence in the Main Theorem also holds weakly and in the Wasserstein distance:
	\begin{cor}\label{cor:weak}
		For each fixed $k$, the image of $\tilde\rho_t^{N-1}\,d\tsig^{N-1}$ under $\mb x\mapsto(x_1,\dots,x_k)$ converges weakly, and in the Wasserstein distance $W_p$ for every $p\ge1$, to $u^{\infty,k}_t\,d\mb x$.
	\end{cor}
	\begin{proof}
		A Gaussian law is determined by its moments, so the convergence of all moments in Theorem~\ref{thm:main} implies weak convergence (\cite[Thm.~30.2]{Billingsley95}). Applying Theorem~\ref{thm:main} to $f=|\mb x|^{2q}$ for every $q$ gives uniform integrability of every power, hence convergence of all absolute moments. Weak convergence together with convergence of the $p$-th moments is equivalent to convergence in $W_p$ \cite[Thm.~6.9]{Villani09}.
	\end{proof}
	
	\medskip

\section{A Derivation Using PDE}\label{sec:PDE}

In the sequel, let us simplify the notation for convenience: fix a positive integer $k$, we write $\mb y=(x_2,\dots,x_k)$, as well as the multi-index notations
\begin{align*}
	\mb y^\al &= x_2^{n_2}x_3^{n_3}\dots x_k^{n_k}\\ 
	|\al|&= n_2+n_3+\dots+n_k.
\end{align*}
We also use $\pt_x$ for $\dfrac{\pt}{\pt x}$ for any dummy variable $x$, and the Cauchy--Euler operator
\begin{align*}
	\mb y\pt_{\mb y}:= x_2\pt_{x_2}+\dots+x_k\pt_{x_k}
\end{align*}
\subsection{The spherical Laplacian and the heat operator}
We now consider the explicit formula for the spherical Laplacian in real coordinates, which plays an important role in establishing Theorem \ref{thm:main}. Still considering the sphere centered at the origin $S^{N-1}(a)$, from the general theory of smooth manifolds, there exists a self-adjoint Laplacian $\Del_{S^{N-1}(a)}$ with respect to the normalized volume measure $\sig^{N-1}$. The \emph{heat kernel} $\rho^{N-1}_{a,t}(\mb z_0,\mb z)$ at a base point $\mb z_0$, time $t$, and a position $\mb z$ is the fundamental solution to the \emph{spherical heat equation}
\begin{align}
	\frac{\pt}{\pt t}\, \rho^{N-1}_{a,t}(\mb z_0,\mb z)=\frac{1}{2}\Del_{S^{N-1}(a)}\, \rho^{N-1}_{a,t}(\mb z_0,\mb z),\label{eqn:heat}
\end{align}
with the initial condition
\begin{align*}
	\lim_{t\downarrow 0} \int_{\mb z\in S^{N-1}(a)} \rho^{N-1}_{a,t}(\mb z_0,\mb z)\,f(\mb z)\,d\sig^{N-1}_a(\mb z)=f(\mb z_0),\qquad f\in C^\infty(S^{N-1}(a)).
\end{align*}

It is known that $\rho^{N-1}_{a,t}(\mb z_0,\mb z)$ is symmetric in $\mb z_0$ and $\mb z$ and only depends on the (spherical) distance between $\mb z_0$ and $\mb z$. 

\medskip

Although there is a formula of the Laplacian $\Delta_M$ in local coordinates on any smooth manifold $M$, for the problem we are considering, we will not use it. Rather, we will consider the sphere $S^{N-1}(a)$ as embedded in $\R^N$, and write $\Del_{S^{N-1}(a)}$ explicitly as differential operator in the variables $x_1,\dots,x_N$ as follows.

\begin{prop}[{\cite{PeSen}}]\label{prop:spherical-laplacian}
	The Laplacian on the sphere $S^{N-1}(a)$ is given by
	\begin{align*}
		\Del_{S^{N-1}(a)}
		&= \frac{1}{a^2}\left(r^2\Del_{\R^N}-(N-2)(r\pt_r)_N-(r\pt_r)^2_N\right)\big|_{S^{N-1}(a)}. 
	\end{align*}
	where $r=\sqrt {x_1^2+x_2^2+\dots+x_N^2}$ is the radial variable which satisfies 
	$$\displaystyle r\pt_r=(r\pt_r)_N := x_1\pt_{x_1}+x_2\pt_{x_2}+\dots+x_N\pt_{x_N}$$
	from the Cauchy--Euler formula. In particular, if $a=\sqrt N$, then 
	\begin{align*}
		\Del_S=\Del_{S^{N-1}(\sqrt N)} = \left.\left[\Del_{\R^N}-\left(1-\dfrac{2}{N}\right)(r\pt_r)_N - \dfrac{1}{N}(r\pt_r)_N^2\right]\right|_{S^{N-1}(\sqrt N)}
	\end{align*}
\end{prop}

We should understand 
the formula in Proposition \ref{prop:spherical-laplacian} as follows:  To compute the spherical Laplacian of $f\in C^2(S^{N-1}(a))$,  first, we extend a function $f$ on $S^{N-1}(a)$ to a smooth function defined on the neighborhood of the sphere in $\R^N$ and consider it as a function in $\R^N$, then we apply the differential operators given in the formula, and restrict the result back to the sphere. It is known that the result is independent of the choice of extension.

We consider polynomials in this study, because a polynomial $f$ as a smooth function with an explicit formula in $\R^N$, can be represented with the same formula when restricted to the sphere $S^{N-1}(a)\hookrightarrow \R^N$.

One important consequence of Prop.~\ref{prop:spherical-laplacian} is the following result when we apply the Laplacian on polynomial of $k$ variables $x_1,\dots,x_k$ where $k<N$.

\begin{prop}[{\cite[Remark after Lem.~3.2]{Coifman1971}}]\label{prop:maps-of-sph-lap}
	For integers $k,N$ with $k<N$, and any polynomial $f$ depending only on variables $x_1,\dots,x_k$, there exists a \emph{unique} polynomial $g$ also depending only on $x_1,\dots,x_k$ such that 
	\begin{align*}
		\left.\Del_{S^{N-1}(a)} f\right|_{S^{N-1}(a)}=g\big|_{S^{N-1}(a)}.
	\end{align*}
\end{prop}

As a result, we can consider $\Del_S$ as a map from the space of all polynomials of $k$ real variables $x_1,\dots,x_k$, denoted by $\mathcal P^k(\R)$, to itself, as long as $k<N$. We also note that the uniqueness will not be true if $k=N$ (the existence is still true, however).

\medskip

Next, let us consider the following important linear functional.
\begin{defi}\label{defi:heat}
	Let $\mb p=(\sqrt N,0,\dots,0)$ be the north pole of $S^{N-1}(\sqrt N)$ and $\rho^{N-1}_t(\mb x)=\rho^{N-1}_{\sqrt N,t}(\mb p,\mb x)$ the heat kernel based at $\mb p$. 
	
	Define the \emph{heat operator} {$H^{N-1}_t\colon C^\infty(S^{N-1}(\sqrt N)) \to \R$} by the formula:
	\begin{align*}
		H^{N-1}_t\,f = \int_{S^{N-1}(\sqrt N)} f(\mb x)\,\rho^{N-1}_t(\mb x)\,d\sig^{N-1}(\mb x).
	\end{align*} 
	In other words, $H^{N-1}_tf$ is the integral of $f$ over the sphere against the heat kernel measure $\rho^{N-1}_t\,d\sig^{N-1}$ based at $\mb p$.
\end{defi} 

The following theorem, which goes back to an unpublished thesis of Langlands \cite{Langlands}, will be of our interest; for completeness we include a short proof.

\begin{thm}\label{thm:powerseries}
	For any $t\ge0$, $k<N$, let $f$ be a polynomial of $x_1,\dots,x_k$ considered as a function on $S^{N-1}(\sqrt N)$, and $\mb p=(\sqrt N,0,\dots,0)$ be the north pole of the sphere. We have
	\begin{align*}
		H^{N-1}_tf=\left(e^{\frac{t}2\Del_S} f\right)(\mb p):=\left.\sum_{n=0}^\infty \frac{1}{n!}\left(\frac{t}{2}\Del_S\right)^n f\,\right|_{\mb x=\mb p},
	\end{align*}
	Here, the operator on the right-hand side is a power series of differential operators applied term-wise to $f$ and evaluates the final result at the north pole $\mb p$.
\end{thm}

\begin{proof}
	The restriction to $S^{N-1}(\sqrt N)$ of a polynomial of degree at most $\ell$ is a finite sum of spherical harmonics of degrees $\le \ell$ (see, e.g., \cite[Ch.~IV]{Stein-Weiss71}). Now, a spherical harmonic of degree $j$ is an eigenfunction of $\Del_{S^{N-1}(\sqrt N)}$ with eigenvalue $-j(j+N-2)/N$. 
	
	Hence $f$ lies in a finite-dimensional $\Del_S$-invariant space on which $\Del_S$ is diagonalizable, and on this space the heat semigroup $(P_tf)(\mb x)=\int f(\mb y)\,\rho^{N-1}_t(\mb x,\mb y)\,d\sig^{N-1}(\mb y)$ is equal to the matrix exponential $e^{\frac t2\Del_S}f$, which is given by the (finite-dimensional and convergent) power series. Since $H^{N-1}_tf=(P_tf)(\mb p)$, evaluating at $\mb p$ gives the claim.
\end{proof}

\subsection{The mean of $x_1$ and the variable $\tilde x_1$}

To understand why we need to redefine the first coordinate, changing from $x_1$ to $\tx_1=\tx_1 (t,N)$ in Theorem~\ref{thm:main},
we first see what would happen if we apply the heat operator $H^{N-1}_t$ corresponding to the heat kernel $\rho^{N-1}_t$ based at the north pole $\mb p$ of the non-shifted sphere $S^{N-1}(\sqrt N)$ (i.e., centered at the origin) on the monomial $f(\mb x)=x_1$.

\begin{prop}\label{prop:mean-x1}
	For $t>0$ and a natural number $N$, if $x_1$ is the coordinate on the centered sphere $S^{N-1}(\sqrt N)$, then we have 
	\begin{align*}
		H^{N-1}_t\,x_1=m(t,N) := \sqrt N\,\exp\left\{\dfrac{t}{2}\left(-1+\dfrac{1}{N}\right)\right\}.
	\end{align*}
\end{prop}

\begin{proof}
	We use the power series form of $H^{N-1}_tf$ in Theorem \ref{thm:powerseries} and the second formula of $\Del_S$ in Proposition \ref{prop:spherical-laplacian}. Notice that $(r\pt_r)_N\,x_1 = x_1$, so
	\begin{align*}
		\Del_Sx_1 &= 
		\frac{1}{N}\left(r^2\Del_{\R^N}-(r\pt_r)^2_N-(N-2)(r\pt_r)_N\right)\big|_{S^{N-1}(\sqrt N)} x_1\\
		& = \dfrac{1}{N}(-x_1-(N-2)x_1) = \left(-1+\dfrac{1}{N}\right)x_1.
	\end{align*}
	As a result,
	\begin{align*}
		\sum_{n=0}^\infty \frac{1}{n!}\left(\frac{t}{2}\Del_{S^{N-1}(\sqrt N)}\right)^n x_1
		&= \sum_{n=0}^\infty \dfrac{1}{n!}\left[\dfrac{t}{2}\left(-1+\dfrac{1}{N}\right)\right]^n x_1 = \exp\left\{\dfrac{t}{2}\left(-1+\dfrac{1}{N}\right)\right\}\, x_1.
	\end{align*}
	Hence, evaluating at the north pole ${\mb p=(\sqrt N,0,\dots, 0)}$ gives us
	\begin{align*}
		H^{N-1}_tx_1 =\left. \exp\left\{\dfrac{t}{2}\left(-1+\dfrac{1}{N}\right)\right\}\, x_1\right|_{x_1=\sqrt N} = m(t,N).
	\end{align*}
\end{proof}

Applying the same idea, we also have the following:
\begin{cor}
	For $t>0$, we have 
	\begin{align*}
		H^{N-1}_t\,x_k= 0.
	\end{align*}
	for $1<k\le N$.
\end{cor}

If integrating a function over a measure is like finding its mean with respect to that measure, then Proposition \ref{prop:mean-x1} suggests that as $N\to\infty$, the mean $\langle x_1\rangle=m(t,N)$ also tends to infinity. It is then natural to introduce the new variable $\tilde x_1=\tilde x_1(t,N)=x_1+m(t,N)$, so that unshifted variable $x_1$ has mean $0$.

More specifically, with this change of variable, we can check that for the shifted sphere $\tilde S^{N-1}(\sqrt N)$, since $\pt_{x_1}=\pt_{\tx_1}$, the Laplacian $\Delta_{\tilde S^{N-1}(\sqrt N)}$ can be computed using the formula
\begin{align}
	\Delta_{\tilde S^{N-1}(\sqrt N)} =\left.\left[\Del_{\R^N}-\left(1-\dfrac{2}{N}\right)(\tr\pt_{\tr})_N-\dfrac{1}{N}(\tr\pt_{\tr})^2_N\right]\right|_{ \tilde S^{N-1}(\sqrt N)}\label{eqn:shift-Laplacian}
\end{align}
with the new radial variable is $\tr = \sqrt{{\tilde x_1}^2+\dots+x_N^2}$ and the new Cauchy--Euler formula
\begin{align*}
	\tr\pt_{\tr}:=\tx_1\pt_{\tx_1} + x_2\pt_{x_2} + \dots + x_N\pt_{x_N}
\end{align*}

The sphere $\tilde S^{N-1}(\sqrt N)$ is a translate of $S^{N-1}(\sqrt N)$, so we define the heat operator $\tilde H^{N-1}_t$ as in Definition \ref{defi:heat}, with the heat kernel measure $\tilde\rho_t^{N-1}$ based at $\tilde{\mb p}$:
\begin{align*}
	\tilde H^{N-1}_t\,f = \int_{\tilde S^{N-1}(\sqrt N)} f(\mb x)\,\tilde\rho^{N-1}_t(\mb x)\,d\tsig^{N-1}(\mb x).
\end{align*}

Then we can use the same argument of Proposition \ref{prop:mean-x1} to get $\tilde H^{N-1}_t \tx_1 = m(t,N)$ and check that $\tilde H^{N-1}_t 1=1$, so for $x_1$ (not $\tx_1$) we have
\begin{align*}
	\tilde H^{N-1}_t x_1 = \tilde H^{N-1}_t [\tx_1-m(t,N)] = m(t,N)-m(t,N)=0.
\end{align*}

\subsection{Action of the heat operators on polynomials and the limiting heat kernel $u^\infty_t$.}\label{sec:heat-poly} In this section, we fix a $t>0$ and a positive integer $k$. We consider $\R^k$ as a sub-manifold of $\R^\infty$ in the usual sense, that is, a point in $\R^k$ is a point in $\R^\infty$ with coordinate $x_j=0$ for $j>k$.

For simplicity, we simply write $m$ for $m(t,N)$ when necessary, keeping in mind that $m$ depends on $N$. We recall that the Laplacian $\Delta_{\tilde S^{N-1}(\sqrt N)}$ of the shifted sphere $\tilde S^{N-1}(\sqrt N)$ can be computed using the formula
\begin{align*}
	\Delta_{\tilde S^{N-1}(\sqrt N)} =\left.\left[\Del_{\R^N}-\left(1-\dfrac{2}{N}\right)(\tr\pt_{\tr})_N-\dfrac{1}{N}(\tr\pt_{\tr})^2_N\right]\right|_{ \tilde S^{N-1}(\sqrt N)}
\end{align*}

\medskip

Since $\Del_{\tilde S^{N-1}(\sqrt N)}$ is self-adjoint with respect to $d\tsig^{N-1}$, the density $\tilde\rho^{N-1}_t$ satisfies (in the coordinate $\tx_1$) that $\pt_t\,\tilde\rho^{N-1}_t = \frac12\Del_{\tilde S^{N-1}(\sqrt N)}\tilde\rho^{N-1}_t$. However, in the formal computations below we write functions in the ambient coordinates $\mb x\in\R^N$ and use the Lebesgue measure of $\R^k$ (and Dirac distributions), with respect to which the operators $x_j\pt_{x_j}$ are \emph{not} self-adjoint.

\medskip

It is true, fortunately, that Theorem \ref{thm:powerseries} holds for $\tilde S^{N-1}(\sqrt N)$ and $\tilde\rho^{N-1}_t\,d\tsig^{N-1}$, with the evaluation at the north pole being  $\tilde {\mb p} = \tilde {\mb p} (t,N)=(\sqrt N-m, 0,\dots,0)$ instead of $(\sqrt N,0,\dots,0)$. 

\medskip

Consider $f$ a polynomial depending only on variables $x_1,x_2,\dots,x_k$ where $k<N$. In the standard coordinates $\mb x$ of $\R^N$, we have
\begin{align}\label{eqn:shift-Langlands}
	\int_{\tilde S^{N-1}(\sqrt N)} f(\mb x)\,\tilde\rho^{N-1}_t(\mb x)\,d\tsig^{N-1}(\mb x) = \left(e^{\frac{t}2\Del_{\tilde S^{N-1}(\sqrt N)}} f\right)(\tilde {\mb p})
\end{align}

\medskip

From Theorem \ref{thm:powerseries}, with $N>k$ and $\tilde S^{N-1}(\sqrt N)$ replacing $S^{N-1}(\sqrt{N})$, we see that
\begin{align*}
	\tilde H^{N-1}_t f = \left(e^{\frac{t}2\Del_{\tilde S^{N-1}(\sqrt N)}} f\right)(\tilde{\mb p})= \sum_{n=0}^\infty \dfrac{1}{n!}\left(\dfrac{t}{2} \Delta_{\tilde S^{N-1}(\sqrt N)}\right)^n f(\tilde{\mb p})
\end{align*}
which involves a power series of the Laplacian $\Delta_{\tilde S^{N-1}(\sqrt N)}$. As seen from \eqref{eqn:shift-Laplacian}, the formula of the Laplacian involves the Laplacian $\Del_{\R^N}$ on $\R^N$ and the Cauchy--Euler operator $\tr\pt_{\tr}$. We note that the restriction of $f$ to $\tilde S^{N-1}$ is unique (this is not true when $k=N$, see \cite{Doan2024}), so we don't need to include the notation of restriction to $\tilde S^{N-1}(\sqrt N)$ if we assume the underlying assumption that the action is on polynomials of fixed number of variables. Hence, we can drop ``$\left.\right|_{\tilde S^{N-1}(\sqrt N)}$'' in the formula of $\Del_{\tilde S^{N-1}(\sqrt N)}$. 

Also, since $f$ depends only on $x_1,\dots,x_k$ with $k<N$, the differential operators $x_j\pt_{x_j}$ and $\pt_{x_j}$ with $j>k$ do not have any effect on $f$, so we can write 
\begin{align}
	\Del_{\tilde S^{N-1}(\sqrt N)} &f = \left[\left(\pt_{\tx_1}^2+\pt_{x_2}^2+\dots+\pt_{x_k}^2\right)-
	\left(1-\dfrac{2}{N}\right)\left(\tx_1\pt_{\tx_1}+x_2\pt_{x_2}+\dots+x_k\pt_{x_k}\right)\right.\notag\\
	&-\left.\dfrac{1}{N}\left(\tx_1\pt_{\tx_1}+x_2 \pt_{x_2}+\dots+x_k\pt_{x_k}\right)^2\right] f=:D_{N,k}f\label{eqn:restrict-shift-Laplacian}
\end{align}

\medskip

Formally, we have the Dirac-delta function $\delta(\mb x)$ that satisfies 
$$\int_{\R^N}f(\mb x)\,\del(\mb x-\mb x_0)\,d\mb x=f(\mb x_0)$$ 
then Equation \eqref{eqn:shift-Langlands} implies
\begin{align*}
	\int_{\tilde S^{N-1}(\sqrt N)} f(\mb x)\,\tilde\rho^{N-1}_t(\mb x)\,d\tsig^{N-1}(\mb x) = \int_{\R^N}\left(e^{\frac{t}2\Del_{\tilde S^{N-1}(\sqrt N)}} f\right)(\mb x)\;\del(\mb x-\tilde {\mb p})\,d\mb x
\end{align*}

Since $f$ depends only on $x_1,\dots, x_k$, we have 
\begin{align*}
	\int_{\tilde S^{N-1}(\sqrt N)} &f(\mb x)\,\tilde\rho^{N-1}_t(\mb x)\,d\tsig^{N-1}(\mb x)\\ 
	&= \int_{\R^k}\left(e^{\frac{t}2D_{N,k}} f\right)(x_1,\dots,x_k)\;\del\left(\tilde x_1-\sqrt N\right)\del(x_2)\dots\del(x_k)\,dx_1\dots\,dx_k
\end{align*}

Taking the adjoint of $e^{\frac{t}2D_{N,k}}$ with respect to $d\mb x$ we obtain
\begin{align}
	\int_{\tilde S^{N-1}(\sqrt N)} &f(\mb x)\,\tilde\rho^{N-1}_t(\mb x)\,d\tsig^{N-1}(\mb x)\notag\\ 
	&= \int_{\R^k} f(x_1,\dots,x_k)\;\left\{e^{\frac{t}2D_{N,k}^*}\del(\tilde x_1-\sqrt N)\del(x_2)\dots\del(x_k)\right\}\,dx_1\dots\,dx_k\label{eqn:induced-eqn}
\end{align}

Let $\tilde\rho^{(N-1),\,k}_t(\tx_1,x_2,\dots,x_k)=\left\{e^{\frac{t}{2}D_{N,k}^*}\del\left(\tilde x_1-\sqrt N\right)\del(x_2)\dots\del(x_k)\right\}$.

Equation~\eqref{eqn:induced-eqn} suggests that the image of the heat kernel measure $\tilde\rho^{N-1}_t\,d\tsig^{N-1}$ under the projection $\mb x\mapsto(x_1,\dots,x_k)$, 
applying to polynomials in $\mathcal P^k(\R)$, is given (in the formal sense of distributions) by $\tilde\rho^{(N-1),\,k}_t(\tx_1,x_2,\dots,x_k)\,dx_1\dots dx_k$. Here, the restricted kernel $\tilde\rho^{(N-1),\,k}_t$ satisfies the differential equation:
\begin{align}
	\pt_t\, \tilde\rho^{(N-1),\,k}_t(\tx_1,x_2,\dots,x_k) = \dfrac{1}{2}{D_{N,k}}^*\tilde\rho^{(N-1),\,k}_t(\tx_1,x_2,\dots,x_k)\label{eqn:restricted-heat}
\end{align}
With the initial condition at $t=0$
$$\tilde\rho^{(N-1),\,k}_0(\tx_1,x_2,\dots,x_k)=\del\left(\tx_1-\sqrt N\right)\del(x_2)\dots\del(x_k).$$
\medskip

Now, we write everything in the standard $d\mb x$ coordinates instead of $d\mbtx$ by defining a new measure
\begin{align*}
	\nu^{(N-1),\,k}_t(x_1,x_2,\dots,x_k) = \tilde\rho^{(N-1),\,k}_t(\tx_1,x_2,\dots,x_k)
\end{align*}
We recall that $x_1=\tx_1 - m=\tx-m(t,N)$, so at $t=0$,  we have 
$$\nu^{(N-1),\,k}_0(x_1,\dots,x_k)= \tilde\rho^{(N-1),\,k}_0(x_1+m(0,N),x_2,\dots,x_k) = \del(x_1)\dots\del(x_k)$$
Equation~\eqref{eqn:restricted-heat} can now be written as
\begin{align}
	\dfrac{\pt}{\pt t}\nu^{(N-1),\,k}_t(x_1,\dots,x_k) &= \dfrac{\pt \tilde\rho^{(N-1),\,k}_t}{\pt \tx_1} \, \dfrac{d\tx_1}{dt} + \dfrac{\pt}{\pt t} \tilde\rho^{(N-1),\,k}_t(\tx_1,x_2,\dots,x_k)\notag\\
	& = \dfrac{\pt \nu^{(N-1),\,k}_t}{\pt x_1}\,\dfrac{d m}{dt} + \dfrac{1}{2}D_{N,k}^*\nu^{(N-1),\,k}_t(x_1,\dots,x_k)\notag\\
	\dfrac{\pt}{\pt t}\nu^{(N-1),\,k}_t(x_1,\dots,x_k) & = \dfrac{1}{2}\left[m \left(-1+\dfrac{1}{N}\right)\dfrac{\pt }{\pt x_1} + D_{N,k}^*\right] \nu^{(N-1),\,k}_t(x_1,\dots,x_k)\label{eqn:adj-kernel}
\end{align}

\medskip

With respect to $d\mb x$, $x_j^*=x_j$ and $\pt_{x_j}^*=-\pt_{x_j}$
we have 
\begin{align*}
\left(x_j\pt_{x_j}\right)^*=-\pt_{x_j} x_j=-x_j\pt_{x_j}-1\tag{here, $1$ means the identity operator}
\end{align*}
so
\begin{align*}
	\left(\tx_1\pt_{\tx_1}+\mb y\pt_{\mb y}\right)^*=\left(\tx_1\pt_{\tx_1}\right)^* + \sum_{j=2}^k \left(x_j\pt_{x_j}\right)^* = -k - m\pt_{x_1} - B
\end{align*}
where $B=\sum_{j=1}^k x_j\pt_{x_j}$ (not $r\pt_r=(r\pt_r)_N$), which means 
\begin{align*}
	\left\{\left(\tx_1\pt_{\tx_1}+ \mb y\pt_{\mb y}\right)^2\right\}^* 
	&= \left(-k - m\pt_{x_1} - B\right)^2\\
	&= k^2 + m^2 \pt_{x_1}^2 +B^2 
	+ 2km\pt_{x_1} + 2kB
	+ m\pt_{x_1}B + mB\pt_{x_1}
\end{align*}
Write $A=\tx\pt_{x_1}+\mb y\pt_{\mb y}$ and $a_N:=1-\dfrac1N$, so that $A^*=-(k+m\pt_{x_1}+B)$. These identities give
\begin{align*}
	&-a_N m\pt_{x_1}+D_{N,k}^*\\ 
	&= m\left(-1+\dfrac{1}{N}\right)\pt_{x_1}+\sum_{j=1}^k\pt_{x_j}^2-\left(1-\dfrac{2}{N}\right)A^*-\dfrac{1}{N}\left(A^*\right)^2\\
	&= m\left(-1+\dfrac{1}{N}\right)\pt_{x_1}+\sum_{j=1}^k\pt_{x_j}^2+\left(1-\dfrac{2}{N}\right)\left(k+m\pt_{x_1}+B\right) 
	-\dfrac{1}{N}\left(k+m\pt_{x_1}+B\right)^2\\
	&= \left(1-\dfrac{m^2}{N}\right)\dfrac{\pt^2}{\pt x_1^2}+\dfrac{\pt^2}{\pt x_2^2}+\dots+\dfrac{\pt^2}{\pt x_k^2} + B + k -\dfrac{1}{N}d_{N,k}
\end{align*}
where 
\begin{align*}
	d_{N,k} = k^2+2k+m\pt_{x_1} + 2k\,m\pt_{x_1} + 2B + 2kB + B^2 
	+ m\pt_{x_1}B + m\,B\pt_{x_1}.
\end{align*}

Now, as $\lim\limits_{N\to\infty}\dfrac{m(t,N)^2}{N}=e^{-t}$ and $\lim\limits_{N\to\infty}\dfrac{m(t,N)}{N}=0$, the operator 		
\begin{align*}
	m(t,N)\left(-1+\dfrac{1}{N}\right)\pt_{x_1}+D_{N,k}^*	
\end{align*}
converges coefficientwise 
when $N$ is large to the following operator:
\begin{align*}
	\left(1-e^{-t}\right)\dfrac{\pt^2}{\pt x_1^2}+\dfrac{\pt^2}{\pt x_2^2}+\dots+\dfrac{\pt^2}{\pt x_k^2} + \sum_{j=1}^k x_j\dfrac{\pt}{\pt x_j}+k.
\end{align*}

If we let $u^{\infty,k}_t$ be the limiting kernel of $\nu^{(N-1),\,k}_t$ by formally letting $N\to \infty$, from Equation~\eqref{eqn:adj-kernel}, the function $u^{\infty,k}_t$ satisfies the differential equation 
\begin{align}
	\dfrac{\pt}{\pt t} u^{\infty,k}_t = \dfrac{1}{2}\left\{\left(1-e^{-t}\right)\dfrac{\pt^2}{\pt x_1^2}+\dfrac{\pt^2}{\pt x_2^2}+\dots+\dfrac{\pt^2}{\pt x_k^2} + \sum_{j=1}^k x_j\dfrac{\pt}{\pt x_j}+k\right\} u^{\infty,k}_t	
	\label{eqn:limiting-kernel}
\end{align}
with the initial condition $u^{\infty,k}_0 = \del(x_1)\dots\del(x_k)$.

The solution $u^{\infty,k}_t$ to Equation~\eqref{eqn:limiting-kernel} with the above initial condition, whose derivation can be found in Appendix~\ref{apd:pde} is
\begin{align*}
	u^{\infty,k}_t (x_1,\dots,x_k) &= c_{t,k} \exp\left(-\dfrac{x_1^2}{2(1-e^{-t}-te^{-t})}-\dfrac{x_2^2+x_3^2+\dots+x_k^2}{2(1-e^{-t})}\right)
\end{align*}
where $c_{t,k}=(2\pi)^{-k/2}(1-e^{-t}-te^{-t})^{-1/2}(1-e^{-t})^{(1-k)/2}$.

The measures $u^{\infty,k}_t\,dx_1\dots dx_k$ satisfy the following compatibility condition: for any integers $0<j\le k$, the marginal of $u^{\infty,k}_t\,dx_1\dots dx_k$ on $\R^j$ (the image under the projection onto the first $j$ coordinates) is $u^{\infty,j}_t\,dx_1\dots dx_j$ (integrate out $x_{j+1},\dots,x_k$).

Thus, by the Kolmogorov Extension Theorem (\cite[Sect.~3.5]{Bog2007}) there exists a probability measure $u^{\infty}_t$ on $\R^\infty$ whose marginal on each $\R^k$ is $u^{\infty,k}_t\,dx_1\dots dx_k$, as described in the Main Theorem~\ref{thm:main}.

\begin{rem}\hfill
	\begin{enumerate}[(1)]
		\item Although the argument presented above is not yet a proof, since one needs to clarify the subtlety in distribution theory with the introduction of the Dirac-delta functions and the adjoints of unbounded operators $x_j$ and $\pt_{x_j}$, it does present a natural way to derive the formula of the kernel in the Main Theorem. Since we are only interested in the action of the heat operators $\tilde H^{N-1}_t$ on polynomials, i.e., integrating polynomials over the heat kernels on $\tilde S^{N-1}(\sqrt N)$, the heuristic can be made rigorous by writing it in weak form: see 
		the proof in Section~\ref{sec:proof-weak}. A second proof, using the Baker--Campbell--Hausdorff formula and explicit combinatorial calculation, is given in Section~\ref{sec:proof-eigen} and Appendix~\ref{apd:proof4.1and4.2}.
		\item The assumption that $f$ depends on finitely many variables in \eqref{eqn:shift-Langlands} is crucial, without which the adjoint operator of $\Del_{S^{N-1}(\sqrt N)}$ will not have a meaningful large-$N$ limit. Therefore, the conclusion of Theorem~\ref{thm:main} that the heat kernels on the spheres $S^{N-1}(\sqrt N)$ converge to a Gaussian measure $u^\infty_t$ on $\R^\infty$ should be understood correctly as the statement that the marginal of the limit on any $\R^k$ is a Gaussian probability measure.
	\end{enumerate}
\end{rem}

\section{A Proof via the Weak Formulation}\label{sec:proof-weak}
\subsection{The weak form heat operator}\label{sec:weak}

The computation in Section~\ref{sec:heat-poly} manipulates densities with Dirac deltas and adjoints of the unbounded operators $x_j$ and $\pt_{x_j}$. As we only integrate \emph{polynomials}, we can avoid all of this by differentiating the numbers $\tilde H^{N-1}_tf$ in $t$ directly. We again use the simplification of notations
\begin{align*}
	a_N:=1-\dfrac1N,\qquad m=m(t,N)=\sqrt N\,e^{-a_Nt/2}
\end{align*}
so that
\begin{align*}
	\dfrac{\pt m}{\pt t}=-\dfrac{a_N}{2}\,m,\quad \dfrac{m^2}{N}=e^{-a_Nt}.
\end{align*}
We also write $\Del_S$ and $\Del_{\tilde S}$ instead of $\Del_{S^{N-1}(\sqrt N)}$ and $\Del_{\tilde S^{N-1}(\sqrt N)}$, respectively, keeping in mind the dependence on $N$ for the operators. Recall $x_1=\tx_1-m$ and that $D_{N,k}$ is the operator in \eqref{eqn:restrict-shift-Laplacian}.

\begin{prop}[Weak heat equation on $\tilde S^{N-1}(\sqrt N)$]\hfill\label{prop:weak-heat}\\
	Let $N>k$ and $f\in\mathcal P^k(\R)$. Then $t\mapsto\tilde H^{N-1}_tf$ is differentiable and
	\begin{align}\label{eqn:weak-heat}
		\dfrac{d}{dt}\tilde H^{N-1}_tf=\tilde H^{N-1}_t\left(\mathcal L_{N,t}f\right),\qquad \textrm{where\qquad } \mathcal L_{N,t}:=\dfrac12D_{N,k}+\dfrac{a_N}{2}\,m\,\pt_{x_1}.
	\end{align}
	Moreover $\tilde H^{N-1}_0f=f(\mb 0)$.
\end{prop}
\begin{proof}
	Let 
	$g_t(\tx_1,x_2,\dots,x_k):=f(\tx_1-m(t,N),x_2,\dots,x_k)=f(x_1,\dots,x_k)$. Recalling the (non-shift) operator $H^{N-1}_t$ in Definition~\ref{defi:heat}, 
	we can verify that $\tilde H^{N-1}_tf=H^{N-1}_tg_t$. 
	
	By Theorem~\ref{thm:powerseries}, $H^{N-1}_sg=\left(e^{\frac s2\Del_S}g\right)(\mb p)$ for $s>0$ and $g$ in the finite-dimensional space $\mathcal P^k_{\le\ell}(\R)$, which is invariant under $\Del_{S^{N-1}(\sqrt{N})}$  as a form of \eqref{eqn:restrict-shift-Laplacian} (see also Proposition~\ref{prop:maps-of-sph-lap}). Hence, $(s,g)\mapsto H^{N-1}_sg$ is smooth in $s$ with 
	\begin{align*}
		\frac{\pt}{\pt s}H^{N-1}_sg=\frac12H^{n-1}_s(\Del_Sg)
	\end{align*}
	(which is just the heat equation), and $t\mapsto g_t$ is a smooth curve in $\mathcal P^k_{\le\ell}(\R)$ with 
	\begin{align*}
		\frac{\pt}{\pt t}g_t=-\frac{\pt m}{\pt t}\,\frac{\pt}{\pt\tx_1}g_t=\frac{a_N}{2}m\,\frac{\pt}{\pt\tx_1}g_t
	\end{align*}
	
	The product rule gives $$\frac{d}{dt}H^{N-1}_t(g_t)=H^{N-1}_t\left(\frac12\Del_Sg_t+\frac{a_N}2m\,\frac{\pt}{\pt\tx_1}g_t\right).$$
	
	Translating back by $\tx_1=x_1+m$, and using that $\Del_Sg_t$ is $D_{N,k}f$ written in the variable $\tx_1$ by \eqref{eqn:restrict-shift-Laplacian}, we get \eqref{eqn:weak-heat}. The last claim is the initial condition $H_0=\del_{\mb p}$ together with $m(0,N)=\sqrt N$.
\end{proof}
Equation \eqref{eqn:weak-heat} is the exact statement that $\tilde\rho^{N-1}_t\,d\tsig^{N-1}$, viewed in the moving coordinate $x_1=\tx_1-m$, solves the heat equation: it is the weak form of \eqref{eqn:adj-kernel}, with $\mathcal L_{N,t}$ in place of the adjoint operator. 

\medskip

The following computation is the weak counterpart of the cancellation in the computation of $-a_Nm\pt_{x_1}+D_{N,k}^*$ in Sect.~\ref{sec:PDE}.

\begin{lem}\label{lem:L-infty}
	With $\mb y\pt_{\mb y}:=\sum_{j=2}^kx_j\pt_{x_j}$,
	\begin{align*}
		\mathcal L_{N,t}=\dfrac12\Big[&\Big(1-\dfrac{(x_1+m)^2}{N}\Big)\pt_{x_1}^2+\sum_{j=2}^k\pt_{x_j}^2-a_N\,x_1\pt_{x_1}\\
		&-\Big(1-\dfrac2N\Big)\mb y\pt_{\mb y}-\dfrac1N\left(\mb y\pt_{\mb y}\right)^2-\dfrac2N(x_1+m)\pt_{x_1}\,\mb y\pt_{\mb y}\Big].
	\end{align*}
	In particular, all terms containing the large factor $m=m(t,N)\sim O(\sqrt N)$ without the factor $1/N$ are canceled. 
	
	Furthermore, if we write
	\begin{align*}
		\mathcal L_{N,t}=\mathcal L^\infty_t+R_{N,t},\qquad \textrm{where}\qquad \mathcal L^\infty_t:=\dfrac12\left[(1-e^{-t})\pt_{x_1}^2+\sum_{j=2}^k\pt_{x_j}^2-\sum_{j=1}^kx_j\pt_{x_j}\right],
	\end{align*}
	and $R_{N,t}$ is the remainder, then for each $\ell$ and $T$, the operator norm of $R_{N,t}$ on $\mathcal P^k_{\le\ell}(\R)$ (with respect to the $\ell^1$-norm of the coefficients in the monomial basis) is at most $C(T,k,\ell)\,N^{-1/2}$ for some fixed constant $C(T,k,\ell)$ independent of $N$, and for all $t\in[0,T]$.
\end{lem}
\begin{proof}
	By \eqref{eqn:restrict-shift-Laplacian} and $(\tx_1\pt_{x_1})^2=\tx_1^2\pt_{x_1}^2+\tx_1\pt_{x_1}$,
	\begin{align*}
		D_{N,k}=\pt_{x_1}^2+\sum_{j\ge2}\pt_{x_j}^2-\Big(1-\dfrac2N\Big)(\tx_1\pt_{x_1}+\mb y\pt_{\mb y})-\dfrac1N\left(\tx_1^2\pt_{x_1}^2+\tx_1\pt_{x_1}+2\tx_1\pt_{x_1}\mb y\pt_{\mb y}+(\mb y\pt_{\mb y})^2\right).
	\end{align*}
	The terms containing $m$ without a factor $\dfrac{\tx_1^2}{N}$ or $\dfrac{\tx_1}{N}\mb y\pt_{\mb y}$ are $$-\left(1-\frac2N\right)m\pt_{x_1}-\frac1Nm\pt_{x_1}=-a_Nm\pt_{x_1},$$ 
	which is cancelled by the term $+a_Nm\pt_{x_1}$ in $2\mathcal L_{N,t}$ following the similar computation in Sect~\ref{sec:PDE}. The remaining terms give the displayed formula of $\mathcal L_{N,t}$, using $(1-\frac2N)+\frac1N=a_N$ for the coefficient of $x_1\pt_{x_1}$. 
	
	For the second claim, expand 
	\begin{align*}
		1-\frac{(x_1+m)^2}N=\left(1-\frac{m^2}N\right)-\frac{2m}Nx_1-\frac{x_1^2}N \textrm{\quad and\quad} \frac{m^2}N=e^{-a_Nt}=e^{-t}e^{t/N}
	\end{align*}	 
	The difference to the limiting coefficients is $O(T/N)+O(N^{-1/2})$, because $\frac mN\lesssim N^{-1/2}$, $\frac1N\le N^{-1/2}$ and $a_N-1=-\frac1N$, and all operators involved map $\mathcal P^k_{\le\ell}$ to itself with norms bounded in terms of $k,\ell$.
\end{proof}

\subsection{The time-inhomogeneous Ornstein--Uhlenbeck process and the proof}
The operator $\mathcal L^\infty_t$ introduced in Lem.~\ref{lem:L-infty} is the generator of the time-inhomogeneous Ornstein--Uhlenbeck process
\begin{align*}
	dX_1=-\tfrac12X_1\,dt+\sqrt{1-e^{-t}}\,dW_1,\quad dX_j=-\tfrac12X_j\,dt+dW_j\quad(2\le j\le k),\quad X(0)=\mb 0.
\end{align*}
Here, $W_j$ is the Wiener process in variable $X_j$. The formal adjoint of this generator is $$\frac12\left\{(1-e^{-t})\pt_{x_1}^2+\sum_{j\ge2}\pt_{x_j}^2+\sum_jx_j\pt_{x_j}+k\right\},$$ which is the operator in \eqref{eqn:limiting-kernel}. 

The next lemma is the weak form of \eqref{eqn:limiting-kernel} and its solution.

\begin{lem}\label{lem:gauss-weak}
	For $f\in\mathcal P^k(\R)$ and $t>0$, define 
	$$\Phi_\infty(t)f:=\int_{\R^k}f\,u^{\infty,k}_t\,d\mb x,$$ 
	and 
	$\Phi_\infty(0)f:=f(\mb 0)$. Then $t\mapsto\Phi_\infty(t)f$ is continuous on $[0,\infty)$ and 
	\begin{align*}
		\dfrac{d}{dt}\Phi_\infty(t)f=\Phi_\infty(t)\left(\mathcal L^\infty_tf\right),\qquad t>0.
	\end{align*}
\end{lem}
\begin{proof}
	For $t>0$, $u^{\infty,k}_t$ is smooth with Gaussian decay of all derivatives, locally uniformly in $t>0$, and solves \eqref{eqn:limiting-kernel} (see Appendix~\ref{apd:pde}). 
	
	Integrating $f\,\pt_tu^{\infty,k}_t$ by parts, which is justified by the decay, and using that the operator in \eqref{eqn:limiting-kernel} is the formal adjoint of $\mathcal L^\infty_t$ gives the identity. The moments of the centered Gaussian with variances $v(t)=1-e^{-t}-te^{-t}$ and $w(t)=1-e^{-t}$ are polynomials in $v,w$, which vanish at $t=0$; this gives continuity.
\end{proof}

Fix $T>0$ 
, $k$ and $\ell\ge0$, and let $N>k$. We equip $\mathcal P^k_{\le\ell}(\R)$, the finite-dimensional vector space of all polynomials of $k$ variables $x_1,\dots,x_k$ of degree at most $\ell$, with the $\ell^1$-norm of the monomial coefficients, and equip its dual with the dual norm, and operators on $\mathcal P^k_{\le\ell}(\R)$ with the operator norm. Consider the following linear functionals
\begin{align*}
	\Phi_N(t):=\tilde H^{N-1}_t\big|_{\mathcal P^k_{\le\ell}(\R)},
\end{align*}
and $\Phi_\infty(t)$  as in Lemma~\ref{lem:gauss-weak}. Since the operators $\mathcal L_{N,t}$ and $\mathcal L^\infty_t$ preserve $\mathcal P^k_{\le\ell}(\R)$, Proposition~\ref{prop:weak-heat} and Lemma~\ref{lem:gauss-weak} say that
\begin{align*}
	\Phi_N'(t)=\Phi_N(t)\circ\mathcal L_{N,t},\quad\Phi_\infty'(t)=\Phi_\infty(t)\circ\mathcal L^\infty_t\quad(t>0),\quad\Phi_N(0)=\Phi_\infty(0)=\del_{\mb 0}.
\end{align*}

\begin{proof}[Proof of Theorem~\ref{thm:main}]
	Let $\Psi_N:=\Phi_N-\Phi_\infty$. It is continuous on $[0,T]$, $\Psi_N(0)=0$, and for $t>0$
	\begin{align*}
		\Psi_N'=\Psi_N\circ\mathcal L_{N,t}+\Phi_\infty(t)\circ R_{N,t},
	\end{align*}
	with $R_{N,t}=\mathcal L_{N,t}-\mathcal L^\infty_t$. By Lemma~\ref{lem:L-infty}, $\|\mathcal L_{N,t}\|\le C$ and $\|R_{N,t}\|\le CN^{-1/2}$ for all $t\in[0,T]$ with $C=C(k,\ell,T)$ independent of $N$. Note that $C':=\sup_{t\le T}\|\Phi_\infty(t)\|<\infty$ because Gaussian moments are bounded for $t\le T$. 
	
	Integrating from $\varepsilon$ to $t$ and letting $\varepsilon\to0$ (all terms are bounded and $\Psi_N$ is continuous),
	\begin{align*}
		\|\Psi_N(t)\|\le\int_0^t\left(C\|\Psi_N(s)\|+C'CN^{-1/2}\right)ds,
	\end{align*}
	so Gr\"onwall's inequality gives $\sup_{t\le T}\|\Psi_N(t)\|\le C'CTN^{-1/2}e^{CT}\to0$. This means $\|\Phi_N-\Phi_\infty\|\to 0$ uniformly in $t\in[0,T]$, so $\tilde H^{N-1}_tf\to\int f\,u^{\infty,k}_t\,d\mb x$ for every $f\in\mathcal P^k_{\le\ell}(\R)$, uniformly in $t\in[0,T]$ and with the explicit error $O(N^{-1/2})$.
\end{proof}

\begin{rem}
	In terms of probabilistic understanding of the result, Lem.~\ref{lem:L-infty} says that the process $$\left(\tx_1(s)-m(s,N),x_2(s),\dots,x_k(s)\right)$$ where $\tx_1(s)$ is Brownian motion on $S^{N-1}(\sqrt N)$, with generator $\frac12 \Del_S$ started at the north pole, has a generator that converges on polynomials to the generator $\mathcal L^\infty_t$ of the Ornstein--Uhlenbeck process. The proof above is the moment-level version of the standard principle (convergence of generators implies convergence of laws). For projections of spherical Brownian motion onto a direction $\boldsymbol\theta$ with $\langle\boldsymbol\theta,\mb x^{(N)}\rangle\to u$, the path-level convergence to an Ornstein--Uhlenbeck process is the content of Skipper \cite[Thm.~2.1]{Skipper2010}, proved by applying the Stroock--Varadhan convergence theorem to the generator of the projection. The pole direction treated here lies outside the hypotheses of that theorem, and a path-level statement for the recentered process above would require checking the corresponding hypotheses (note that the diffusion coefficient $1-e^{-s}$ of the limit degenerates at $s=0$), which we do not pursue here.
\end{rem}
\section{A Proof using Explicit Limits of Actions on Monomials}\label{sec:proof-eigen} 
The goal of this section is to give a second, independent proof of the Main Theorem~\ref{thm:main}, based on the explicit diagonalization of the radial part of the Laplacian (Lemmas~\ref{EO-lem:eigenpoly}--\ref{EO_lem:powseries}). In contrast with Section~\ref{sec:proof-weak}, it works with the semigroup at the fixed final time $t$, and it yields exact finite-$N$ formulas for $\tilde H^{N-1}_t$ on monomials.

\subsection{Actions on monomials}

We recall that the action of the heat operator $\tilde H^{N-1}_t$ on a polynomial $f\in\mathcal P^k(\R)$ can be expressed as:
\begin{align*}
	\tilde H^{N-1}_t f = \left.\sum_{n=0}^\infty \dfrac{1}{n!}\left(\dfrac{t}{2}\Del_{\tilde S^{N-1}(\sqrt N)}\right)^nf \right|_{\mbtx=\tilde {\mb p}}.
\end{align*}
where the operator on the right-hand side is a power series applied to $f$ term-wise. 

Let us now fix a $t>0$ and a positive integer $k$. We recall that $\mathcal P^k(\R)$ denotes the vector space of all polynomials in the real variables $x_1,\dots,x_k$. Let $\mathcal P^k_{\le \ell}(\R)$ denote the space of all polynomials in $\mathcal P^k(\R)$ whose degrees are at most $\ell$. Then,
\begin{align*}
	\mathcal P^{k}(\R) = \bigcup_{\ell=0}^\infty \mathcal P^{k}_{\le \ell}(\R).
\end{align*}

The action of the operator $\Del_{\tilde S^{N-1}(\sqrt N)}$ on $f$, with $N>k$ can be further expressed as $\Del_{\tilde S^{N-1}(\sqrt N)} f = D_{N,k}f$ with $D_{N,k}$ defined in \eqref{eqn:restrict-shift-Laplacian}.

Since $\Del_{\tilde S^{N-1}(\sqrt N)}f$ does not increase the degree or introduce a new variable when acting on $f\in\mathcal{P}^k_{\le \ell}(\R)$, we see that $\Del_{\tilde S^{N-1}(\sqrt N)}$ is a self-map on each finite-dimensional space $\mathcal P^k_{\le \ell}(\R)$.

\medskip

Let us simplify the notations further for convenience: write $x$ and $\tilde x$ for $x_1$ and $\tilde x_1$, respectively, and write $\Del=\sum_{j=1}^k \pt_{x_j}^2$.

Let us introduce the following operators:
\begin{align}
	D&=\pt_{x_1}^2-\left(1-\dfrac{2}{N}\right)\tx_1\pt_{x_1}-\dfrac{1}{N}\left(\tx_1\pt_{x_1}\right)^2\label{eqn:op-D}\\
	& = \pt_{x}^2 - \left(1-\dfrac{2}{N}\right) \tx \pt_x - \dfrac{(\tx \pt_x)^2}{N}\notag
\end{align}
and \begin{align}
	E=E^{N,k}
	&=\sum_{j=2}^k\pt_{x_j}^2-\left(1-\dfrac{2}{N}\right)\sum_{j=2}^kx_j\pt_{x_j}-\dfrac{1}{N}\left(\sum_{j=2}^kx_j\pt_{x_j}\right)^2\label{eqn:op-E}\\
	&=\sum_{j=2}^k\pt _{x_j}^2-\left(1-\dfrac{2}{N}\right)\mb y\pt_{\mb y}-\dfrac{(\mb y\pt_{\mb y})^2}{N} \notag
\end{align}
(Note that $D$ is independent of $k$ while $E$ depends on $k$.)

Then, we can easily check that on $\mathcal P^k_{\le \ell}(\R)$,
\begin{align*}	
	\Del_{\tilde S} &= \Del - \left(1-\dfrac{2}{N}\right)(\tilde x\pt_x+\mb y\pt_{\mb y}) -\frac{1}{N} (\tilde x\pt_x+\mb y\pt_{\mb y})^2\\ 
	&= D + E - \dfrac{2}{N}\,\tilde x\pt_x\,\mb y\pt_{\mb y}.  
\end{align*}

We aim to establish the following results in Appendix~\ref{apd:proof4.1and4.2}.

\begin{thm}\label{thm:x2toxk}
	For any monomial $f(x_2,\dots,x_k)=\mb y^\al=x_2^{n_2}\dots x_k^{n_k}$,
	\begin{align*}
		\lim_{N\to\infty}&\tilde H^{N-1}_t\, f =\lim_{N\to\infty}\left.\left(e^{tE^{N,k}/2}\, f\right)\right|_{\mb y=\mb 0}\\
		&=\left[2\pi (1-e^{-t})\right]^{\frac{1-k}{2}}\int_{\R^{k-1}} f(\mb y)\,\exp\left(-\dfrac{|\mb y|^2}{2(1-e^{-t})}\right)\,d\mb y
	\end{align*}
	
	Here, $|\mb y|^2=x_2^2+\dots+x_k^2$. In particular, for any $j$ where $2\le j\le k$,
	\begin{align*}
		\lim_{N\to\infty}\tilde H^{N-1}_t \left(x_j^{n}\right) = 
		\begin{cases}
			\dfrac{n!}{2^{n/2}(n/2)!}\, (1-e^{-t})^{n/2},\quad&\text{$n$ even}\\
			0,\quad&\text{$n$ odd}\\
		\end{cases}
	\end{align*}
\end{thm}

In other words, $\tilde H^{N-1}_t g$ tends toward having the same moments as a Gaussian centered at $0$ with variance $1-e^{-t}$ for any monomial of $x_2,\dots,x_k$.

\begin{thm}\label{thm:x1}
	For any monomial $f(x_1)=x_1^n$,
	\begin{align*}
		\lim_{N\to\infty}\tilde H^{N-1}_t f &= \lim_{N\to\infty} \left.\left(e^{tD/2} f\right)\right|_{\tx_1=\sqrt N}\\
		&=\dfrac{1}{\sqrt{2\pi \,(1-e^{-t}-te^{-t})}}\int_\R x_1^n \exp\left(-\dfrac{x_1^2}{2(1-e^{-t}-te^{-t})}\right)\,dx_1 \\
		&=\begin{cases}
			\dfrac{n!}{2^{n/2}(n/2)!}\, (1-e^{-t}-te^{-t})^{n/2},\quad&\text{$n$ even}\\
			0,\quad&\text{$n$ odd}\\
		\end{cases}
	\end{align*}
\end{thm}
In other words, $\tilde H^{N-1}_t f $ tends toward having the same moments as a Gaussian centered at $0$ with variance $1-e^{-t}-te^{-t}$.

Note that $\Del_{\tilde S} f=Df$ for $f=f(x_1)$ and $\Del_{\tilde S} f=Ef$ for $f=f(\mb y)$, so the first equalities of Theorems~\ref{thm:x2toxk} and \ref{thm:x1} follow.


\medskip



\subsection{Proof of the Main Theorem}
The proof of Theorem~\ref{thm:main} below, besides assuming Theorems~\ref{thm:x2toxk} and \ref{thm:x1}, also treats the mixed term $-\frac{2}{N}\,\tx\pt_x\,\mb y\pt_{\mb y}$ of $\Del_{\tilde S}$ as a perturbation. Since $D$ contains the large drift $-(1-\frac2N)m\pt_x$ (recall $m\sim\sqrt N$), one has to be careful: with respect to the monomial basis $\{x^a\}$ the norm of $e^{\pm\sigma D}$ grows like $N^{\ell/2}$ on $\mathcal P_{\le\ell}(\R)$, so one cannot bound $e^{-\sigma D}$ and the perturbation separately. Fortunately, the \emph{conjugated} perturbation $e^{-\sigma D}(\tx\pt_x)e^{\sigma D}$ is only of size $O(\sqrt N)$ (Lemma~\ref{lem:conj-size}), which the factor $\frac2N$ in front will compensate. We introduce notation to make this precise.

Fix $\ell\ge0$. Let $\mathcal P_{\le\ell}(x)$ denote the polynomials in the single variable $x$ of degree $\le\ell$, and $\mathcal P^{k-1}_{\le\ell}(\mb y)$ the polynomials in $\mb y=(x_2,\dots,x_k)$ of total degree $\le\ell$. Put
\begin{align*}
	V_\ell := \mathcal P_{\le\ell}(x)\otimes\mathcal P^{k-1}_{\le\ell}(\mb y)\ \supseteq\ \mathcal P^k_{\le\ell}(\R).
\end{align*}
The operators $D$ (acting on the factor $x$), $E$ (acting on the factor $\mb y$) and $\tx\pt_x\,\mb y\pt_{\mb y}$ all preserve $V_\ell$, hence so does $\Del_{\tilde S}=D+E-\frac2N\tx\pt_x\,\mb y\pt_{\mb y}$. We equip $V_\ell$ with the $\ell^1$-norm $\|g\|$ of the coefficients of $g$ in the monomial basis $\{x^a\mb y^\al\}$, and equip the linear operators with the corresponding operator norm; this norm does not depend on $N$, and $\|K_x\otimes K_y\|\le\|K_x\|\,\|K_y\|$. We write
\begin{align*}
	\theta_x:= \tx\pt_x=(x+m)\pt_x,\qquad \theta_{\mb y}:=\mb y\pt_{\mb y}.
\end{align*}

\begin{lem}\label{lem:conj-size}
	Fix $t>0$, $k$ and $\ell$. There is a constant $C=C(t,k,\ell)$ such that for all $N>k$ and all $0\le\sigma\le t/2$,
	\begin{enumerate}[(i)]
		\item $\left\|e^{-\sigma E}\,\theta_{\mb y}\, e^{\sigma E}\right\|_{\mathcal P^{k-1}_{\le\ell}(\mb y)}\le C$.
		\item $\left\|e^{-\sigma D}\,\theta_x\, e^{\sigma D}\right\|_{\mathcal P_{\le\ell}(x)}\le C\sqrt N$,
	\end{enumerate}
\end{lem}

\begin{proof}
	\begin{enumerate}[(i)]
	\item In the monomial basis, the matrix representation of 
	$$E=\sum_{j=2}^k\pt_{x_j}^2-\left(1-\frac2N\right)\theta_{\mb y}-\frac1N\theta_{\mb y}^2$$ 
	has entries bounded independently of $N$, so $\|E\|\le C_1(k,\ell)$ and $\|e^{-\sigma E}\theta_{\mb y}e^{\sigma E}\|\le e^{2\sigma C_1}\|\theta_{\mb y}\|\le e^{tC_1}\ell$.
	
	\item 
	Since $\pt_x\,\tx=\tx\,\pt_x+1$ we have $\pt_x\,\theta_x=(\theta_x+1)\pt_x$, hence for every polynomial $F$ and $q\ge0$,
	\begin{align}\label{eqn:shift-commutation}
		\pt_x^{q}F(\theta_x)=F(\theta_x+q)\,\pt_x^{q}.
	\end{align}
	Moreover, as operators on polynomials in $x$, $D=\pt_x^2+\phi(\theta_x)$ with
	\begin{align*}
		\phi(z)=-\left(1-\dfrac2N\right)z-\dfrac1N z^2,
	\end{align*}
	because $(\tx\pt_x)^2=\theta_x^2$. On $\mathcal P_{\le\ell}(x)$, the operator $\theta_x$ is triangular in the basis $\{\tx^i\}$ with diagonal entries $0,1,\dots,\ell$, so $\prod_{i=0}^\ell(\theta_x-i)=0$ there, and every $\theta_x^r$ with $r>\ell$ is a combination of $\theta^q_x$, $q\le\ell$, with coefficients depending only on $r$ and $\ell$. Therefore every operator in
	\begin{align*}
		\mathcal T:=\operatorname{span}\left\{\theta_x^r\,\pt_x^{2j}\ :\ 0\le r\le\ell,\ 0\le 2j\le\ell\right\}
	\end{align*}
	(restricted to $\mathcal P_{\le\ell}(x)$) is a finite sum $K=\sum c_{rj}\theta^r\pt^{2j}$, and we define
	\begin{align*}
		\|K\|_\sharp:=\inf\left\{\sum_{r,j}|c_{rj}|\,N^{r/2}\ :\ K=\sum_{r,j}c_{rj}\theta^r\pt^{2j}\ \text{on }\mathcal P_{\le\ell}(x)\right\}.
	\end{align*}
	
	\noindent We claim that:
	\begin{enumerate}[(a)]
		\item $\|\cdot\|_\sharp$ is a norm and $\|K\|\le C_2\|K\|_\sharp$ for $K\in\mathcal T$, with $C_2=C_2(\ell)$;
		\item $\mathcal T$ is invariant under $\operatorname{ad}_D:=[D,\cdot\,]$ and $\|[D,K]\|_\sharp\le C_3\|K\|_\sharp$, with $C_3=C_3(\ell)$ independent of $N$.
	\end{enumerate}
	
	For (a), $\|\pt_x\|\le\ell$ and $\|\theta_x\|\le\|x\pt_x\|+m\|\pt_x\|\le\ell(1+\sqrt N)\le2\ell\sqrt N$ since $m\le\sqrt N$, so $\|\theta_x^r\pt_x^{2j}\|\le(2\ell)^{r+2j}N^{r/2}$. 
	
	For (b), it suffices to work with $K=\theta_x^r\pt^{2j}$. By \eqref{eqn:shift-commutation},
	\begin{align*}
		[\pt_x^2,\theta_x^r\pt_x^{2j}]&=\left((\theta_x+2)^r-\theta_x^r\right)\pt_x^{2j+2},\\
		[\phi(\theta_x),\theta_x^r\pt_x^{2j}]&=\theta_x^r\left(\phi(\theta_x)-\phi(\theta_x+2j)\right)\pt_x^{2j}\\
		&=\left(2j\left(1-\frac2N\right)+\dfrac{4j^2}{N}\right)\theta_x^r\pt_x^{2j}+\dfrac{4j}{N}\,\theta_x^{r+1}\pt_x^{2j}.
	\end{align*}
	The first commutator involves only powers $\theta_x^q$ with $q\le r-1$, with coefficients depending on $r$ only, so its $\sharp$-norm is at most $C N^{(r-1)/2}\le CN^{r/2}$. 
	
	In the second one, the linear term $-(1-\frac2N)\theta_x$ of $\phi$ has dropped out of the commutator except for a bounded multiple of $K$, and the term $\frac{4j}{N}\theta_x^{r+1}\pt_x^{2j}$ has $\sharp$-norm at most $CN^{(r+1)/2-1}\le CN^{r/2}$ (if $r+1>\ell$, we first reduce $\theta_x^{r+1}$ as above, which only lowers the powers of $\theta_x$). This proves (b).
	
	\medskip
	
	Now, we see that 
	$$e^{-\sigma D}\theta_x e^{\sigma D}=\sum_{n\ge0}\frac{(-\sigma)^n}{n!}\operatorname{ad}_D^n(\theta_x)$$ 
	(a finite-dimensional, hence convergent, series), and $\|\theta_x\|_\sharp\le N^{1/2}$. 
	
	By (b), $\|\operatorname{ad}_D^n\theta_x\|_\sharp\le C_3^nN^{1/2}$, so $\|e^{-\sigma D}\theta_x e^{\sigma D}\|_\sharp\le e^{C_3\sigma}N^{1/2}$, and (a) gives (ii) with $C=C_2e^{C_3t/2}$.
	\end{enumerate}
\end{proof}


\begin{proof}[Proof of Theorem~\ref{thm:main} assuming Theorems~\ref{thm:x2toxk} and \ref{thm:x1}]
	Any polynomial $f$ in $\mathcal P^k(\R)$ is a finite sum of monomials of the form $x_1^{n_1}x_2^{n_2}\dots x_k^{n_k}$, so it suffices to prove the theorem for monomials of this form. So, for simplicity, we assume for simplicity $f(x_1,x_2,\dots,x_k)= x_1^{n_1}x_2^{n_2}\dots x_k^{n_k}$.
	
	Fix $\ell$ with $f\in\mathcal P^k_{\le\ell}(\R)\subseteq V_\ell$. We have $\Del_{\tilde S}=\mathcal A+P$ on $V_\ell$, where
	\begin{align*}
		\mathcal A:=D+E,\qquad P:=-\dfrac{2}{N}\,\theta_x\otimes\theta_{\mb y}.
	\end{align*}
	The mixed term $P$ is small, but, as explained before Lemma~\ref{lem:conj-size}, it has to be controlled after conjugation by $e^{\sigma\mathcal A}$. We proceed in three steps.
	
	\smallskip
	\noindent $\bullet$ \emph{Step 1:} We claim that the functional $\Lambda_N:=\del_{\tilde{\mb p}}\,e^{t\mathcal A/2}$ is uniformly bounded on $V_\ell$, and converges on every monomial. 
	
	Indeed, $D$ and $E$ commute and act on different variables, so for a monomial $x^a\mb y^\al\in V_\ell$,
	\begin{align*}
		\Lambda_N\left(x^a\mb y^\al\right)=\left[\del_{x=\sqrt N-m}\,e^{tD/2}x^a\right]\left[\del_{\mb y=\mb 0}\,e^{tE/2}\mb y^\al\right].
	\end{align*}
	Once we prove Theorems~\ref{thm:x2toxk} and \ref{thm:x1} later, both factors converge as $N\to\infty$. Since $V_\ell$ has a finite basis of monomials, $\sup_N\|\Lambda_N\|<\infty$, and $\Lambda_N(g)$ converges for every $g\in V_\ell$.

	\noindent $\bullet$ \emph{Step 2:} For $\sigma\in[0,t/2]$ put $W(\sigma):=e^{-\sigma\mathcal A}e^{\sigma\Del_{\tilde S}}$ on $V_\ell$. Then $W(0)=I$ and
	\begin{align*}
		W'(\sigma)=e^{-\sigma\mathcal A}(\Del_{\tilde S}-\mathcal A)e^{\sigma\Del_{\tilde S}}=P(\sigma)\,W(\sigma),\qquad P(\sigma):=e^{-\sigma\mathcal A}Pe^{\sigma\mathcal A}.
	\end{align*}
	Because $D$ and $E$ commute and act on different tensor factors,
	\begin{align*}
		P(\sigma)=-\dfrac2N\left(e^{-\sigma D}\theta_x e^{\sigma D}\right)\otimes\left(e^{-\sigma E}\theta_{\mb y}e^{\sigma E}\right),
	\end{align*}
	so Lemma~\ref{lem:conj-size} gives $\|P(\sigma)\|\le\frac{2}{N}\cdot C\sqrt N\cdot C=C'N^{-1/2}$ uniformly for $\sigma\in[0,t/2]$. By Gr\"onwall's inequality,
	\begin{align*}
		\left\|W(t/2)-I\right\|\le\exp\left(\dfrac{C't}{2\sqrt N}\right)-1\xrightarrow{N\to\infty}0.
	\end{align*}
	
	\smallskip
	\noindent $\bullet$ {Step 3:} Since $e^{t\Del_{\tilde S}/2}=e^{t\mathcal A/2}W(t/2)$,
	\begin{align*}
		\tilde H^{N-1}_tf=\del_{\tilde{\mb p}}\,e^{t\Del_{\tilde S}/2}f=\Lambda_N\,W(t/2)\,f=\Lambda_N\,f+\Lambda_N\,(W(t/2)-I)\,f,
	\end{align*}
	and the last term is bounded by $\sup_N\|\Lambda_N\|\cdot\|W(t/2)-I\|\cdot\|f\|\to0$. Hence 
	$$\lim_{N\to\infty}\tilde H^{N-1}_tf=\lim_{N\to\infty}\Lambda_N(f),$$ 
	i.e., $\tilde H^{N-1}_t$ converges to the product of the limits of $\del_{x=\sqrt N-m}\,e^{tD/2}$ and $\del_{\mb y=\mb 0}\,e^{tE/2}$ on monomials.
	
	\medskip
	
	More explicitly, we have 
	\begin{align*}
		\lim_{N\to\infty}\tilde H^{N-1}_t\,f = \lim_{N\to\infty}\left.\left(e^{tD/2} x_1^{n_1}\right)\right|_{x_1=\sqrt N - m(t,N)}\left.
		\lim_{N\to\infty}\left(e^{tE/2} x_2^{n_2}\dots x_k^{n_k} \right)\right|_{x_2=\dots=x_k=0}
	\end{align*} 
	
	Then, by Theorems~\ref{thm:x1} and \ref{thm:x2toxk}, 
	%
	\begin{align*}
		&\lim_{N\to\infty} \tilde H^{N-1}_t\, f \\
		&=c_{t,k}\int_{\R^{k-1}}\,x_2^{n_2}\dots x_k^{n_k}
		\exp\left(-\dfrac{x_2^2+\dots+x_k^2}{2(1-e^{-t})}\right)
		\,\int_{\R} x_1^{n_1} \exp\left(-\dfrac{x_1^2}{2(1-e^{-t}-te^{-t})}\right) \,dx_1 dx_2\dots dx_k\\
		&=\int_{\R^k} f\,du^{\infty,k}_t
	\end{align*}
	where
	\begin{align*}
		u^{\infty,k}_t (x_1,x_2,\dots,x_k) &= c_{t,k} \exp\left(-\dfrac{x_1^2}{2(1-e^{-t}-te^{-t})}-\dfrac{x_2^2+x_3^2+\dots+x_k^2}{2(1-e^{-t})}\right)
	\end{align*}
	and $c_{t,k}=(2\pi )^{-k/2}(1-e^{-t}-te^{-t})^{-1/2}(1-e^{-t})^{(1-k)/2}$ as we want.
\end{proof}

\appendix
\section{Solution to the limiting PDE in Section~\ref{sec:heat-poly}}\label{apd:pde}
Equation~\eqref{eqn:limiting-kernel} is a parabolic differential equation with a standard method of solution using the Fourier transforms and the characteristic functions. 

The operator on the right-hand side of \eqref{eqn:limiting-kernel} is a sum of operators each acting on a single variable, so we look for a solution of the form $u^{\infty,k}_t = g_1(t,x_1) g_2(t,x_2)\dots g_k(t,x_k)$. After expanding the partial derivatives and dividing both sides by $u^{\infty,k}_t$, we obtain 
\begin{align*}
	\dfrac{1}{g_1}\dfrac{\pt g_1}{\pt t} + \sum_{j=2}^k \dfrac{1}{g_j}\dfrac{\pt g_j}{\pt t} = \dfrac{1}{2}\left\{\left[\dfrac{1-e^{-t}}{g_1}\dfrac{\pt^2 g_1}{\pt x_1^2}+\dfrac{x_1}{g_1}\dfrac{\pt g_1}{\pt x_1}+1\right] + \sum_{j=2}^k\left[\dfrac{1}{g_j}\dfrac{\pt^2 g_j}{\pt x_j^2}+\dfrac{x_j}{g_j}\dfrac{\pt g_j}{\pt x_j}+1\right]\right\}
\end{align*}
It therefore suffices to solve the following two parabolic equations of one variable (a product of solutions then solves \eqref{eqn:limiting-kernel}):
\begin{align}
	g_t = \dfrac{1}{2}((1-e^{-t})g_{xx}+xg_x+g) \label{eqn:parabolic-1}
\end{align}
and
\begin{align}
	g_t = \dfrac{1}{2}(g_{xx}+xg_x+g) \label{eqn:parabolic-2}
\end{align}
We solve \eqref{eqn:parabolic-1} and obtain the solution for \eqref{eqn:parabolic-2} from the same argument. Taking the Fourier transform of both sides of \eqref{eqn:parabolic-1}, where 
\begin{align*}
	\hat g(t,\xi) = \mathcal F(g(t,x)) = \int_\R g(t,x) e^{-i\xi x}\,dx.
\end{align*}
Using the identities $\mathcal F(g_{xx})=-\xi^2\hat g$ and $\mathcal F(xg_x) = -\hat g-\xi\hat g_\xi$, we get
\begin{align}
	\hat g_t= -\dfrac{1}{2}\left[(1-e^{-t})\xi^2\hat g+\xi\hat g_\xi\right] \label{eqn:first-order-pde}
\end{align}
which is a first-order differential equation in $t$ and $\xi$. Introducing the characteristic curve $\xi(t)$ that satisfies $\dfrac{d}{dt}\xi (t)=\dfrac{1}{2}\xi (t)$, which gives $\xi(t) = \xi_0 e^{t/2}$.

Along the characteristic curve, the function $\hat g$ satisfies 
$$\hat g_t = \dfrac{d}{dt}\hat g(t,\xi(t))-\hat g_\xi \xi_t =  \dfrac{d}{dt}\hat g(t,\xi(t))-\dfrac{1}{2}\xi\hat g_\xi$$ 
so substituting back to \eqref{eqn:first-order-pde}, we have 
\begin{align*}
	\dfrac{d}{dt}\hat g(t,\xi(t)) = \dfrac{1}{2}\left[-(e^t-1)\xi_0^2\right]\hat g (t,\xi(t))
\end{align*}

Integrating along the characteristic curve from $0$ to $t$, we obtain  
\begin{align*}
	\hat g(t,\xi) = 
	\hat g(0,\xi_0)\exp\left\{-\dfrac{1}{2}\left[(e^{t}-t-1)\xi_0^2\right]\right\}
\end{align*}
Since $\xi(0)=e^{-t/2}\xi$, we obtain
\begin{align*}
	\hat g(t,\xi) = \hat g(0,\xi e^{-t/2})\exp\left\{-\dfrac{1}{2}(1-e^{-t}-te^{-t})\xi^2\right\}
\end{align*}
The initial condition is $g(0,x)=\del(x)$, so $\hat g(0,\xi) = 1$.

Taking the inverse Fourier transform, we have 
\begin{align*}
	g(t,x) &= \mathcal F^{-1}(\hat g(t,\xi)) =  \dfrac{1}{2\pi} \int_\R \hat g(t,\xi)e^{i\xi x}\,d\xi\\
	&= \dfrac{1}{\sqrt{2\pi(1-e^{-t}-te^{-t})}}\exp\left\{-\dfrac{x^2}{2(1-e^{-t}-te^{-t})}\right\}
\end{align*}

Using the same argument, we obtain the following solution for equation \eqref{eqn:parabolic-2}
\begin{align*}
	g(t,x) = \dfrac{1}{\sqrt{2\pi(1-e^{-t})}}\exp\left\{-\dfrac{x^2}{2(1-e^{-t})}\right\}
\end{align*}

Therefore, the solution to Equation~\eqref{eqn:limiting-kernel} with initial condition $\del(x_1)\dots\del(x_k)$ is
\begin{align*}
	u^{\infty,k}_t (x_1,x_2,\dots,x_k) &= c_{t,k} \exp\left(-\dfrac{x_1^2}{2(1-e^{-t}-te^{-t})}-\dfrac{x_2^2+x_3^2+\dots+x_k^2}{2(1-e^{-t})}\right)
\end{align*}
where $c_{t,k}=(2\pi)^{-k/2}(1-e^{-t}-te^{-t})^{-1/2}(1-e^{-t})^{(1-k)/2}$.

\section{Proofs of Theorems~\ref{thm:x2toxk} and \ref{thm:x1}}\label{apd:proof4.1and4.2}
\begin{lem}[{A special case of the Baker--Campbell--Hausdorff formula, \cite[Lem.~6.10]{Doan2024}}]\label{lem:BCH}
	Given any real square matrices $X$ and $Y$ of the same order, suppose there exists $b\ne 0$ such that $[X,Y]=XY-YX=bY$, then 
	\begin{align*}
		\exp(X+Y) = \exp X\,\exp\left(\dfrac{1-e^{-b}}{b} Y\right)
	\end{align*}
\end{lem}

\begin{proof}[Proof of Theorem~\ref{thm:x2toxk}]
	We can check that for any monomial $\mb y^\al = x_2^{n_2}\dots x_k^{n_k}$, the operator $\mb y\pt_{\mb y}$ satisfies the following properties:
	\begin{align}
		&\mb y\pt_{\mb y} \left(\mb y^\al\right) = |\al| \mb y^\al,\notag\\
		& \left[\pt_{x_j}^2,x_j\pt_{x_j}\right] = 2\pt_{x_j}^2\qquad\text{for $2\le j\le k$}\notag\\
		& \left[-\dfrac{t}{2}\mb y\pt_{\mb y},\dfrac{t}{2}\sum_{j=2}^k\pt_{x_j}^2\right] = t\cdot\left(\dfrac{t}{2}\sum_{j=2}^k\pt_{x_j}^2\right),\label{eqn:comm-relation}
	\end{align}
	
	Since $\mb y\pt_{\mb y}$ does not increase the degree of $\mb y^\al$ and does not introduce new variables beyond $x_2,\dots,x_k$, as a linear operator on the finite dimensional vector space $\mathcal P^k_{\le \ell} (\R)$, it is clear that
	\begin{align*}
		\lim_{N\to\infty} E = \sum_{j=2}^k (\pt_{x_j}^2-x_j\pt_{x_j}) = \sum_{j=2}^k \pt_{x_j}^2-\mb y\pt_{\mb y}
	\end{align*}
	
	The operator on the right-hand side is also known as the Hermite operator. Its exponential can be computed by using Lem.~\ref{lem:BCH} with $X=-\frac t2\mb y\pt_{\mb y}$, $Y=\frac t2\sum_{j=2}^k\pt_{x_j}^2$ and $b=t$, by Eq.~\eqref{eqn:comm-relation}:
	\begin{align*}
		\lim_{N\to\infty} e^{tE/2}&=\exp\left\{\dfrac{t}{2}\left(\sum_{j=2}^k \pt_{x_j}^2-\mb y\pt_{\mb y}\right)\right\} 
		&= \underbrace{\exp\left(-\dfrac{t}{2}\mb y\pt_{\mb y}\right)}_{A}\,\underbrace{\exp\left(\dfrac{1-e^{-t}}{2}\sum_{j=2}^k \pt_{x_j}^2\right)}_B
	\end{align*}
	
	The operator $B$ is the real version of the Segal--Bargmann transform (without complexification) from $L^2(\R^{k-1})$ to $L^2(\R^{k-1},d\mu^{k-1}_{1-e^{-t}})$, where $d\mu^{k-1}_{1-e^{-t}}$ is the Gaussian measure with mean $0$ and variance $1-e^{-t}$ for each variable $x_2,\dots,x_k$ (see, \cite{Doan2024}, Sect.~3, and the references therein). The transform can be expressed as an integral
	\begin{align*}
		(Bf)(y)= \left(2\pi (1-e^{-t})\right)^{-(k-1)/2}\int_{\R^{k-1}} f(z)\,e^{-\frac{|z-y|^2}{2(1-e^{-t})}}\,dz. 
	\end{align*}  
	
	The operator $A$ is a dilation map $(Af)(\mb y) = f\left(e^{-t/2}\mb y\right)$ (see \cite[Sect.~6.1]{Doan2024}). Hence, together when evaluating at $x_2=\dots=x_k=0$, we have 
	\begin{align*}
		\lim_{N\to\infty} &\left.\left(e^{tE/2} f\right) \right|_{\mb y=\mb 0} \\
		&=  \left.\left(2\pi (1-e^{-t})\right)^{-(k-1)/2}\int_{\R^{k-1}} f(\mb z)\,\exp\left\{-\dfrac{\left|\mb z-e^{-t/2}\mb y\right|^2}{2(1-e^{-t})}\right\}\,d\mb z\quad\right|_{\mb y=\mb 0}\\
		&=\left(2\pi (1-e^{-t})\right)^{-(k-1)/2}\int_{\R^{k-1}} f(\mb z)\,\exp\left\{-\dfrac{\left|\mb z\right|^2}{2(1-e^{-t})}\right\}\,d\mb z
	\end{align*}
	which is the conclusion of Thm.~\ref{thm:x2toxk}.
\end{proof}

For the proof of Theorem~\ref{thm:x1}, we need to use another approach, because when $N\to\infty$, the coordinate $\tilde x$ moves further away from $x$, and the evaluation $x=\sqrt N-m$ also tends to infinity. Recall that we are using $x$ and $\tilde x$ in lieu of $x_1$ and $\tilde x_1$, respectively. 

\begin{lem}\label{EO-lem:eigenpoly}
	For any $N\ge 2$, the eigenspaces of operator $D$ are all one-dimensional, each of which is spanned by the following polynomial of degree $n$, namely
	\begin{align*}
		p_n(\tx) = \sum_{j=0}^{\floor{n/2}} \left(- \frac{N}{4}\right)^j \frac{n^{\underline{2j}}}{j! \left(\frac{N}{2} + n - 2\right)^{\underline{j}}}\, \tx^{n - 2j}
	\end{align*}
	corresponding to eigenvalue
	\begin{align*}
		\lambda_n = -n\left(1 + \frac{n-2}{N}\right).
	\end{align*}
	where $z^{\underline{j}}$ denotes the falling factorial $z^{\underline{j}}=z(z-1)(z-2)\dots(z-j+1)$.
\end{lem}

\begin{proof}
	Note that for the monomial $\tilde x^n$, we have $(\tx\pt_x)\, \tx^n = n\tx^n$, so 
	\begin{align*}
		D\,\tilde x^n = n(n-1)\tx^{n-2}+\lam_n\tx^n.
	\end{align*}
	
	Then, we can write 
	\begin{align*}
		p_n(\tx) = \tx^n + a_1\tx^{n-2} + a_2 \tx^{n-4} + \dots + a_{\floor{n/2}}\tx^{n-2\floor {n/2}}.
	\end{align*} 
	Solving $D p_n(\tx) =\lam_n p_n(\tx)$ we obtain the recurrence relation
	\begin{align*}
		&(n-2j+2)(n-2j+1)\, a_{j-1} + \lam_{n-2j}\, a_j= \lam_n\,a_j,\quad 1\le j\le \floor{n/2},\quad a_0=1,
	\end{align*}
	from which we can solve for the coefficient of $\tx^{n-2j}$ as obtained in the statement of the lemma.  
	
	On the space of polynomials in $\tx$ with coefficients in $\mathbb Q(N)$, $D$ is triangular with respect to the basis $\{\tx^n\}$ with the pairwise distinct diagonal entries $\lambda_n$ ($\lambda_n=\lambda_{n'}$ forces $(n-n')(N+n+n'-2)=0$). Hence $D$ is diagonalizable over $\mathbb Q(N)$ on each space of polynomials of degree $\le n$, and there cannot be any more eigenvectors.
\end{proof}

Next, we evaluate the eigen-polynomials at the desired point $\tx = \sqrt{N}$.

\begin{lem}\label{EO_lem:PW}
	For $N\ge3$,
	\begin{align*}
		p_n\left(\sqrt{N}\right) = N^{n/2}\prod_{i=0}^{n-1}\frac{N+i-2}{N+2i-2}
		=\frac{N^{n/2} (N-2)^{\overline{n}}}{2^n\left(\frac{N}{2}-1\right)^{\overline{n}}}.
	\end{align*}
	where $z^{\overline j}$ denotes the rising factorial $z^{\overline j} = z(z+1)\dots(z+j-1)$.
\end{lem}
\begin{proof}
	Breaking up
	\begin{align*}
		n^{\underline{2j}} = 4^j \floor {\frac{n}{2}} ^{\underline{j}}\left(
		\ceil{\dfrac{n}{2}}- \dfrac{1}{2}\right)^{\underline{j}},
	\end{align*}
	we get
	\begin{align*}
		\frac{p_n\left(\sqrt{N}\right)}{N^{n/2}} &= \sum_{j=0}^{\floor{n/2}} (-1)^j \frac{ \floor{\dfrac{n}{2}}^{\underline{j}}}{j!} \frac{\left(\ceil{\dfrac{n}{2}} - \dfrac{1}{2}\right)^{\underline{j}}}{\left(\dfrac{N}{2} + n - 2\right)^{\underline{j}}} \\
		&= \frac{1}{\dbinom{N/2 + n - 2}{\ceil{n/2} - 1/2}}\sum_{j=0}^{\floor{n/2}} (-1)^j \binom{\floor{n/2}}{j} \binom{N/2 + n - 2 - j}{N/2 + \floor{n/2} - 3/2},
	\end{align*}
	assuming that $n$ has a generic value so that the indicated binomial coefficients are defined. Now applying the standard formula
	\begin{align*}
		\sum_{j=0}^a (-1)^j \binom{a}{j} \binom{b-j}{c} = \binom{b - a}{c - a},
	\end{align*}
	valid for $a \in \mathbb N$, $b, c \in \mathbb C$, we get a form of the answer:
	\begin{align*}
		p_n\left(\sqrt{N}\right) = \frac{N^{n/2} \left(\dfrac{N-1}{2}\right)^{\overline{\floor{n/2}}}}
		{\left(\dfrac{N}{2} + n - 2\right)^{\underline{\floor{n/2}}}}.
	\end{align*}
	The product form in the lemma follows by noting that
	\begin{align*}
		\frac{p_{n+1}(\sqrt{N})}{p_n(\sqrt{N})} = \sqrt{N} \cdot \frac{N+n-2}{N + 2n-2}
	\end{align*}
	regardless of the parity of $n$.
\end{proof}

Next, we invert the change-of-basis matrix relating the two bases $\{p_n\}$ and $\{\tx^n\}$.
\begin{lem}\label{EO_lem:Pinv}
	\begin{align*}
		\tx^n = \sum_{j=0}^{\floor{n/2}} \left(\frac{N}{4}\right)^j \frac{n^{\underline{2j}}}{j! \left(\dfrac{N}{2} + n - j - 1\right)^{\underline{j}}} p_{n - 2j}.
	\end{align*}
\end{lem}

\begin{proof}
	Straightforward computation using back substitution, noting that $p_0=1$ and $p_1=\tx$. 
\end{proof}

\begin{lem}\label{EO_lem:powseries}
	Fix a nonnegative integer $j$. As $N\to \infty$, the factor
	\[
	t_0(h) = \frac{(N-2)^{\overline{h}}}{2^h \left(\dfrac{N}{2}-1\right)^{\overline{h}}\left(\dfrac{N}{2}+h\right)^{\overline{j}}},
	\]
	can be expressed, for each fixed $h$, as a power series in $1/N$
	\begin{equation}\label{EO_eq:pow_ser_t}
		t_0(h) = \sum_{\ell \geq 0} \frac{u_\ell(h)}{N^{\ell + j}}
	\end{equation}
	where the coefficients $u_\ell(h)$ are polynomials in $h$ of degree $2\ell$ with leading term
	\begin{equation}
		u_\ell(h) = \frac{(-1)^\ell 2^{j-\ell}}{ \ell!} h^{2\ell} \left(1 + O\left(\frac{1}{h}\right)\right).\notag
	\end{equation}
\end{lem}
\begin{proof}
	The recurrence relation
	\begin{align}
		(N + 2h - 2)(N+2h+2j)\, t_0(h+1) &= (N + h-2)(N+2h)\, t_0(h) \label{EO_eq:rec_t}\\
		t_0(0) &= \frac{1}{\left(\dfrac{N}{2}\right)^{\overline{j}}} \label{EO_eq:init_t}
	\end{align}
	uniquely specifies $t_0(h)$. Writing \eqref{EO_eq:rec_t} in terms of the $u_\ell$ and comparing the coefficients of $N^{2-\ell-j}$, we get (with $u_{-1}=u_{-2}=0$)
	\begin{align*}
		u_{\ell}(h+1) - u_\ell(h) &= (3h-2)\,u_{\ell - 1}(h) - (4h + 2j - 2)\, u_{\ell - 1}(h + 1)\\
		&\quad + 2h(h-2)\,u_{\ell-2}(h) - (2h-2)(2h+2j)\,u_{\ell-2}(h+1)
	\end{align*}
	for $\ell \geq 0$. If $u_{\ell - 1},u_{\ell-2}$ are known, this determines $u_\ell$ up to an additive constant, and the constant is determined by \eqref{EO_eq:init_t}; in particular $u_0\equiv 2^j$. By induction, $u_\ell$ is a polynomial of degree $2\ell$. Comparing the coefficients of $h^{2\ell-1}$ on both sides (the $u_{\ell-2}$-terms have degree $2\ell-2$ and do not contribute) and writing $u_\ell=c_\ell h^{2\ell}+\dots$ gives $2\ell c_\ell=-c_{\ell-1}$, so $c_\ell=(-1)^\ell 2^{j}/(2^\ell\ell!)$.
\end{proof}


\begin{proof}[Proof of Thm.~\ref{thm:x1}]
	Now, since
	\begin{align}\label{EO_eq:e}
		e^{tD/2} p_n = e^{t\lam_n/2} p_n = \exp\left\{-\frac{nt}{2} + \frac{(2n - n^2)t}{2N}\right\} p_n,
	\end{align}
	we can obtain the moment $e^{tD/2} x^n$ by converting $x^n$ to the basis $\{p_n\}$ with Lemma \ref{EO_lem:Pinv}, applying $e^{tD/2}$ using \eqref{EO_eq:e}, and then plugging in $x = \sqrt{N}-m$ using Lemma \ref{EO_lem:PW}:
	\begin{align}
		&\left. e^{tD/2}x^n \right|_{x = \sqrt{N}-m} = \left. e^{tD/2}(\tx-m)^n \right|_{\tx = \sqrt{N}} \\
		&= \left. \sum_{i=0}^{n} \binom{n}{i} (-m)^i \left(e^{tD/2} \tx^{n-i}\right) \right|_{\tx = \sqrt{N}} \notag \\
		&= \left. \sum_{i=0}^{n} \binom{n}{i} (-m)^i e^{tD/2} \left(\sum_{j=0}^{\floor{\frac{n - i}{2}}} \left(\frac{N}{4}\right)^j \frac{(n - i)^{\underline{2j}}}{j! \left(\dfrac{N}{2} + n - i - j - 1\right)^{\underline{j}}} p_{n - i - 2j} \right) \right|_{\tx = \sqrt{N}} \notag \\
		&= \sum_{i=0}^{n} \binom{n}{i} (-m)^i \sum_{j=0}^{\floor{\frac{n - i}{2}}} \left(\frac{N}{4}\right)^j \frac{(n - i)^{\underline{2j}}}{j! \left(\dfrac{N}{2} + n - i - j - 1\right)^{\underline{j}}}\, e^{t \lambda_{n - i - 2j}/2} p_{n - i - 2j}(\sqrt{N})\label{EO_eq:dbl_sum}
	\end{align}
	
	For fixed $n$, this is a rational function of $\sqrt{N}$, $e^{t/2}$, and $e^{t/(2N)}$. To get the limit as $N \to \infty$, we expand $e^{\pm t/(2N)}$ to express $e^{tD/2}$ as a convergent power series in $1/\sqrt{N}$ (leaving $e^{t/2}$ as it is) and claim that
	\begin{enumerate}
		\item All terms with a positive power of $N$ cancel.
		\item The term not involving $N$ is the claimed moment of Theorem \ref{thm:x1}.
	\end{enumerate}
	
	Observe that the exponential factor of each term of \eqref{EO_eq:dbl_sum}, coming from the values of $m$ and $\lam_n$, is
	\begin{align*}
		& \exp\left\{\frac{-N + 1}{2N} t i -(n - i - 2j)\left(1 + \frac{n - i - 2j - 2}{N}\right) \frac{t}{2} \right\} \\
		&= \exp\left\{\frac{t}{2}\left(-n + 2j\right) + \frac{t}{2N}\left[i - (n - i - 2j)(n - i - 2j - 2)\right]\right\}
	\end{align*}
	
	Accordingly, we fix $j$. Write
	\[
	i - (n-i - 2j)(n-i-2j-2) = -i^2 + k_1 i + k_2,
	\]
	where the $k_i$ are constants that will not matter. We let $c_n^{(j)}$ denote the coefficient of $e^{t (-n + 2j)/2}$ in $e^{tD/2} (\tx-m)^n|_{\tx=\sqrt N}`$, then
	\[
	c_n^{(j)} = \frac{N^j}{4^j j!}\sum_{i=0}^{n - 2j} (-1)^i \binom{n}{i} \frac{N^{i/2}(n - i)^{\underline{2j}}}{\left(\dfrac{N}{2} + n - i - j - 1\right)^{\underline{j}}} e^{(-i^2 + k_1 i + k_2)\frac{t}{2N}}\ p_{n - i - 2j}(\sqrt{N}).
	\]
	By the change of variable $h = n - i - 2j$, we get (for some constants $k_1'$, $k_2'$)
	\begin{align*}
		c_n^{(j)} &= \frac{(-1)^n\, N^{n/2}}{4^j j!}\sum_{h=0}^{n - 2j} (-1)^h \binom{n}{h + 2j} \frac{\left(h + 2j\right)^{\underline{2j}}}{\left(\dfrac{N}{2} + h + j - 1\right)^{\underline{j}}} e^{(-h^2 + k'_1 h + k'_2)\frac{t}{2N}} \frac{p_{h}(\sqrt{N})}{N^{h/2}} \\
		&= \frac{(-1)^n N^{n/2}}{4^j j!}\sum_{h=0}^{n - 2j} (-1)^h \binom{n}{h + 2j} \frac{\left(h + 2j\right)^{\underline{2j}}}{\left(\dfrac{N}{2} + h + j - 1\right)^{\underline{j}}} \frac{ (N-2)^{\overline{h}}}{2^h\left(\dfrac{N}{2}-1\right)^{\overline{h}}}\, e^{(-h^2 + k'_1 h + k'_2)\frac{t}{2N}}.
	\end{align*}
	With a bit of rearrangement of factorials, we get
	\begin{equation}
		c_n^{(j)} = \frac{(-1)^n\, N^{n/2}\, n!}{4^j j! (n - 2j)!} \sum_{h=0}^{n - 2j} (-1)^h \binom{n - 2j}{h} \frac{(N-2)^{\overline{h}}}{2^h \left(\dfrac{N}{2}-1\right)^{\overline{h}}\left(\dfrac{N}{2}+h\right)^{\overline{j}}} e^{\left(-h^2 + k_1'h + k_2'\right)\frac{t}{2N}}.
	\end{equation}
	
	Now, from Lem~\ref{EO_lem:powseries} examining the power-series expansion
	\begin{align*}
		c_n^{(j)} &= \frac{(-1)^n\, N^{n/2}\, n!}{4^j\,j!\,(n - 2j)!} \sum_{h=0}^{n - 2j} (-1)^h \binom{n - 2j}{h} \frac{(N-2)^{\overline{h}}}{2^h \left(\dfrac{N}{2}-1\right)^{\overline{h}}\left(\dfrac{N}{2}+h\right)^{\overline{j}}} \exp\left\{\left(-h^2 + k_1'\,h + h_2'\right)\,\frac{t}{2N}\right\} \\
		&= \frac{(-1)^n N^{n/2} n!}{4^j j! (n - 2j)!} \sum_{h=0}^{n - 2j} (-1)^h \binom{n - 2j}{h} \left[ \sum_{\ell \geq 0} \frac{u_\ell(h)}{N^{\ell + j}} \sum_{i \geq 0} \frac{t^i}{2^i i! N^i} \left(-h^2 + k_1'\, h + k_2'\right)^i\right],
	\end{align*}
	we see that the terms with a nonnegative power of $N$ are those with
	\begin{equation}\label{EO_eq:likj}
		\ell + i \leq n/2 - j.
	\end{equation}
	Moreover, each such term has degree at most $2\ell + 2i$ as a polynomial in $h$. Thus, when the inequality \eqref{EO_eq:likj} is strict, the term vanishes upon taking the $(n - 2j)$th finite difference. In particular, when $n$ is odd the equality case cannot occur (the exponent $n/2-j-\ell-i$ is then a half-integer, and the terms with nonnegative exponent have $\ell+i<n/2-j$), so every such term is killed and the limit is $0$. Hence we can restrict our attention to the equality case in \eqref{EO_eq:likj} and to terms of leading order in $h$, getting
	
	\begin{align*}
		\lim_{N \to \infty} c_{n}^{(j)}
		&= \frac{n!}{4^j j!} \sum_{i = 0}^{n/2 - j} \text{coef}_{\, h^{n - 2j - 2i}} [u_{n/2 - j - i}(h)]  \cdot \frac{(-t)^i}{2^i i!} \\
		&= \frac{n!}{4^j j!} \sum_{i = 0}^{n/2 - j} \frac{(-1)^{n/2 - j - i}2^{i + 2j - n/2}}{ (n/2 - j - i)!} \cdot \frac{(-t)^i}{2^i i!} \\
		&= \frac{(-1)^{n/2 - j} n!}{2^{n/2} j! (n/2 - j)!} \sum_{i = 0}^{n/2 - j} \binom{n/2 - j}{i} t^i \\
		&= \frac{(-1)^{n/2 - j} n!}{2^{n/2} j! (n/2 - j)!} (1 + t)^{n/2 - j}.
	\end{align*}
	
	Now
	\begin{align*}
		\lim_{N \to \infty} e^{tD/2}\, (\tx-m)^n|_{\tx=\sqrt N} &= \sum_{j = 0}^{n/2} e^{-\frac{t}{2}(n - 2j)} \lim_{N \to \infty} c_{n}^{(j)} \\
		&= \frac{n!}{2^{n/2} (n/2)!} \sum_{j=0}^{n/2} \binom{n/2}{j} \left(-(1 + t)e^{-t}\right)^{n/2-j} \\
		&= \frac{n!}{2^{n/2} (n/2)!} \left(1 - (1 + t)e^{-t}\right)^{n/2},
	\end{align*}
	as desired.
\end{proof}

\unappendix
%

\bibliographystyle{plain}
\bibliography{refs.bib}
\end{document}